\documentclass[journal,twoside]{IEEEtran}

\usepackage[hyphens]{url}
\usepackage{times} 
\usepackage{amsmath}
\usepackage{dsfont}
\usepackage{mathrsfs}
\usepackage{amsfonts}
\usepackage{amssymb}
\usepackage{amsthm}
\usepackage{graphicx}
\usepackage{color}
\usepackage{algorithm}
\usepackage[noend]{algpseudocode}
\usepackage{epsfig} 
\usepackage{xcolor}
\usepackage{cite}

\definecolor{purple}{rgb}{0.6,0,0.8}

\newcommand{\plim}[1]{\underset{#1\to\infty}{\textup{p-lim}}}
\newcommand{\T}{\intercal} 
\newcommand{\delAttack}{\hat{\delta}} 
\newcommand{\AL}{\underline{A}}
\newcommand{\AK}{\bar{A}}

\newtheorem{theorem}{Theorem}[section]
\newtheorem{cor}[theorem]{Corollary}

\newtheorem{asmp}[theorem]{Assumption}
\newtheorem{lm}[theorem]{Lemma}
\newtheorem{df}[theorem]{Definition}

\newcommand{\remove}[1]{}
\newcommand{\add}[1]{#1}

\allowdisplaybreaks

\begin{document}
\title{Detecting Generalized Replay Attacks\\ via Time-Varying Dynamic Watermarking}

\author{Matthew~Porter,~Pedro~Hespanhol,~Anil~Aswani,~Matthew~Johnson-Roberson,~and~Ram~Vasudevan
\thanks{This work was supported by a grant from Ford Motor Company via the Ford-UM Alliance under award N022977 and by UC Berkeley Center for Long-Term Cybersecurity.}%
\thanks{M. Porter and R. Vasudevan are with the Department of Mechanical Engineering, University of Michigan, Ann Arbor, MI 48103 USA (e-mail:~matthepo@umich.edu;~ramv@umich.edu).}%
\thanks{A. Aswani and P. Hespanhol are with the Department of Industrial Engineering and Operations Research, University of California Berkeley, Berkeley, CA 94720 USA (e-mail:~aaswani@berkeley.edu;~pedrohespanhol@berkeley.edu).}%
\thanks{M. Johnson-Roberson is with the Department of Naval Architecture, University of Michigan, Ann Arbor, MI 48103 USA (e-mail:~mattjr@umich.edu).}%
}

\maketitle

\begin{abstract}
Cyber-physical systems (CPS) often rely on external communication for supervisory control or sensing.
Unfortunately, these communications render the system vulnerable to cyber-attacks.
Attacks that alter messages, such as replay attacks that record measurement signals and then play them back to the system, can cause devastating effects.
Dynamic Watermarking methods, which inject a private excitation into control inputs to secure resulting measurement signals, have begun addressing the challenges of detecting these attacks, but have been restricted to linear time invariant (LTI) systems.
Though LTI models are sufficient for some applications, other CPS, such as autonomous vehicles, require more complex models.
This paper develops a linear time-varying (LTV) extension to previous Dynamic Watermarking methods by designing a matrix normalization factor to accommodate the temporal changes in the system.
Implementable tests are provided with considerations for real-world systems.
The proposed method is then shown to be able to detect generalized replay attacks both in theory and in simulation using a LTV vehicle model.
\end{abstract}
\section{Introduction}
Cyber-physical systems (CPS) combine both networked computing and sensing resources with physical control systems in an effort to increase efficiency, manage complexity, or provide convenience.
Whether it is industrial control applications or smart devices, CPS require secure networked communications to operate safely and correctly.
Malicious attacks on such systems can cause devastating results \cite{Langner2011,Abrams2008,lee2014german,Lee2016AnalysisGrid}.
CPS are often protected by traditional cyber security tools, but these methods are insufficient due to the addition of networked physical infrastructure. \cite{Sandberg2015,Cardenas2008}.
A growing body of work has started to address these challenges by developing new detection algorithms, analyzing potentially stealthy attack models, and finding ways of reducing the effect of attacks.
One particular detection method, Dynamic Watermarking, has been shown to detect various attack models while making few assumptions about system structure.
Despite these developments, detection algorithms, including Dynamic Watermarking, have only focused on CPS that can be modeled as linear time invariant (LTI) systems.
While LTI models can be sufficient for steady state or slow moving applications, many emerging CPS such as autonomous vehicles require models that change over time.
This paper develops methods to accommodate such CPS by extending Dynamic Watermarking to linear time-varying (LTV) systems.

\subsection{Attack Models}
Attacks are divided into three categories: \emph{denial of service} (DOS) \emph{attacks}, where the control or measurement signal is stopped,  \emph{direct attacks}, where the plant, actuators or sensors are physically attacked, and \emph{deception attacks}, where the control or measurement signal are altered. \cite{Cardenas2008survive}.
DOS attacks can be detrimental, but are trivial to detect if they stop all communication.
Furthermore, when only a portion of communication is stopped, their effects can be minimized using graceful degradation \cite{Saurabh2009dos}.
The result of direct attacks often causes anomalies in the measurement signal and can therefore be detected by methods used to detect deception attacks.
Consequentially, this work focuses on the detection of deception attacks.

A variety of deception attacks have been proposed. 
The simplest deception attacks add noise using arbitrary or random strategies \cite{Liu2011}.
On the other hand, bias injection attacks, the attacker injects a constant bias into the system \cite{Teixeira2012models}, while routing attacks send measurement signals through a linear transform \cite{Ferrari2017isolation}.
Other deception attacks attempt to decouple the system such that the measurements are unaltered while certain states of the system are attacked \cite{smith2011}.
For instance, zero-dynamics attacks take advantage of un-observable states or remove the effects of their attacks in the resulting measurement signal \cite{Teixeira2012models}, and replay attacks involve an attacker replaying recorded measurements while possibly altering control as well \cite{Teixeira2012models}.

The amount of knowledge of the system dynamics and detection scheme along with the capability of the attacker to alter certain signals necessary to carry out these attacks varies greatly.
While random, bias injection, routing, and replay attacks do not require any knowledge of the underlying system dynamics, decoupling and zero-dynamics attack require almost full knowledge. 
This knowledge can be difficult to obtain for non-insider attackers but it is not impossible \cite{yuan2015,Umsonst2019}.
Nonetheless, this work focuses on a generalization of a replay attack due to the simplicity of implementation and because it has already been applied during real-world attacks \cite{Langner2011}.
Furthermore, we consider attacks that only alter measurement signals, since many of the systems we care about use local controllers while operating using externally received measurements.

\subsection{Attack Detection Algorithms}
The \emph{measurement residual}, defined as the difference between the measurement and the expected measurement, is used by most detection schemes.
For each detector, a metric based on the measurement residual is generated.
If at any time the metric exceeds a user-defined threshold, the detector raises an alarm.
Generally, these metrics can be separated into two categories: those that only observe the system, called \emph{passive methods}, and those that alter the system while observing, called \emph{active methods}.
While passive methods do not degrade control performance, active methods accept a small amount of performance degradation in exchange for the ability to detect more complex attacks  \cite{Porter2019,Weerakkody2017exposing,weerakkody2016info}.
These categories can be further subdivided into \emph{stateless} metrics, which only consider the current measurement residual, and \emph{stateful} metrics, which rely on previous measurement residuals as well. 

\subsubsection{Passive Methods}
The $\chi^2$ detector's metric is the inner product of the normalized measurement residual, which follows a $\chi^2$ distribution. 
Due to its simplicity, the  $\chi^2$ detector has been studied in several works \cite{Mo2010,Mo2012,Kwon2013,hashemi2019}.
Though the $\chi^2$ is widely used, it is a stateless detector.
Two stateful alternatives are the cumulative sum (CUSUM) detector and the multivariate exponentially weighted moving average (MEWMA) detector.
When comparing these stateful detectors to the $\chi^2$ detector, it has been shown that the stateful detectors can often provide stronger guarantees on detection while the $\chi^2$ detector boasts both simpler implementation and generally takes less time to detect attacks\cite{Murguia2016CUSUMSensors,Umsonst2018}.
While passive detectors can detect random attacks, they are unable to detect more sophisticated attacks such as replay attacks.
In addition, they have only been developed for LTI systems.

\subsubsection{Active Methods}
Most active methods fall into one of two categories: \emph{moving target defense}, which change system parameters to keep attackers from obtaining the current configuration, and \emph{watermarking-based methods}, which encrypt measurement signals with a watermark that is added to the control input.

The concept of moving target defenses is a topic of continued interest for the field of cyber security and includes randomizing the order of code execution and physical memory storage locations \cite{jajodia2011moving}. 
In CPS, moving target defense can take the form of switching between redundant measurements \cite{Teixeira2012revealing,Rahman2014MovingTD,Tian2017hidden,Giraldo2019,Kanellopoulos2019}, altering control strategy \cite{Teixeira2012revealing,Kanellopoulos2019}, or by changing plant dynamics \cite{Teixeira2012revealing,Rahman2014MovingTD,Tian2017hidden,Weerakkody2015moving,Schellenberger2017auxiliary,Ghaderi2019,Griffioen2019}. 
Switching measurement signals works well when an attacker is only hacking a few measurements, but otherwise performs similar to passive methods.
Altering the control strategy is arguably similar to watermarking-based methods and can allow for detection of most attack models except zero-dynamics attacks.
While some methods alter the physical plant dynamics directly \cite{Teixeira2012revealing,Rahman2014MovingTD,Tian2017hidden}, others append the plant dynamics with an auxiliary system with possibly more complex dynamics \cite{Weerakkody2015moving,Schellenberger2017auxiliary,Ghaderi2019,Griffioen2019}.
Despite the consideration of more complex dynamics for the auxiliary systems, moving target defense has only been applied to systems that have LTI dynamics. 
Although complex dynamics cause the behavior of the test metric to change in time, methods for selecting a time-varying threshold involve hand tuning.
Altering plant dynamics can allow for detection of all attack models, but the method makes certain assumptions about the system.
Note, for the auxiliary systems it is assumed that the plant has secure knowledge of its own state, which does not account for vulnerable networked sensors.
Also, when an auxiliary system is not used, it is assumed that the plant dynamics are changeable.

The introduction of a watermark was first proposed as a way of making the $\chi^2$ detector robust to replay attacks \cite{Mo2009} and other more advanced attacks \cite{Weerakkody2014}.
Here, the watermark takes the form of independent identically distributed (IID) Gaussian noise that is added to the control input.
Robustness to replay attacks is then achieved by properly selecting the watermark covariance, while the $\chi^2$ detector itself remains unchanged.
Dynamic Watermarking uses a metric that relies on both the covariance of the residuals and the correlation between the residuals and the watermark. 
The covariance of the watermark is allowed to be an arbitrary symmetric full rank matrix \cite{Satchidanandan2017,Hespanhol2017,satchidanandan2017defending,Hespanhol2018,rubio2016event}.
In these works, the metric uses the measurement residuals contained in a temporally sliding window.
Guarantees of detection are then made as the window size tends to infinity.
Extensions to a limited subset of nonlinear systems have been implemented \cite{Ko2016,satchidanandan2017defending}, but otherwise Dynamic Watermarking has been limited to LTI systems.
Though the addition of the watermark causes a degradation in system performance, the degradation can be minimized \cite{Mo2014detecting,Hosseini2016optimal}.
Other work has considered allowing the watermark signal to be auto-correlated \cite{Mo2015} or to have distributions that are not Gaussian \cite{Satchidanandan2020design,hespanhol2019sensor}.
Furthermore, other forms of watermarks include intentional package drops \cite{Ozel2017packet,Weerakkody2017packet}, using parameterized transforms on measurements \cite{Ferrari2017isolation,Ferrari2017detection,Teixeira2018multiplicative}, and B-splines added to feed forward inputs \cite{Romagnoli2019}.
Though Dynamic Watermarking is unable to detect zero-dynamics attacks, it does not require the assumption of changeable plant dynamics or locally secure knowledge of plant state.
This paper focuses on Dynamic Watermarking as described in Hespanhol et al. \cite{Hespanhol2017} due to its ability to be applied to a wide range of LTI systems including both fully and partially observable systems.

\subsection{Contributions}

The contributions of this paper are threefold.
First, the tests used in Hespanhol et al. \cite{Hespanhol2017} are extended to LTV systems.
This is done using a carefully designed matrix normalization factor to accommodate the temporal changes in the system.
These tests are then proven to detect generalized replay attacks.
Second, a model is developed for time-varying generalized replay attacks.
Third, LTV Dynamic Watermarking is applied to a simulated system to provide proof of concept.

The remainder of this paper is organized as follows. 
Section \ref{sec:Notation} introduces notation.
Section \ref{sec:inspire} reviews the methods in Hespanhol et al. \cite{Hespanhol2017} to motivate the need for LTV Dynamic Watermarking.
Asymptotic guarantees and implementable tests for LTV Dynamic Watermarking are provided in Sections \ref{sec:LTV_theory} and \ref{sec:LTV_practical} respectively. 
Simulated results are presented in Section \ref{sec:sim}.
The appendix covers statistical background for the proofs in this paper\add{ in addition to a few larger equations that have been removed from proofs to improve readability}.
\subsection{Notation}\label{sec:Notation}
This section breifly introduces the notation used in this paper.
The 2-norm of a vector $x$ is denoted $\|x\|$.
Similarly, the 2-norm of a matrix $X$ is denoted $\|X\|$.
The trace of a matrix $X$ is denoted tr$(X)$.
Zero matrices of dimension $i\times j$ are denoted $0_{i\times j}$, and in the case that $i=j$, the notation is simplified to $0_i$.
Identity matrices of dimension $i$ are denoted $I_i$.

The Wishart distribution with scale matrix $\Sigma$ and $i$ degrees of freedom is denoted $\mathcal{W}(\Sigma,i)$ \cite[Section 7.2]{anderson2003}.
The multivariate Gaussian distribution with mean $\mu$  and covariance $\Sigma$ is denoted $\mathcal{N}(\mu,\Sigma)$.
The chi-squared distribution with $i$ degrees of freedom is denoted $\chi^2(i)$. 
The expectation of a random variable $a$ is denoted $\mathds{E}[a]$.
The probability of an event $E$ is denoted $\mathds{P}(E)$. 
Given a sequence of random variables $\{a_i\}_{i=1}^\infty$, convergence in probability is denoted $\text{p-lim}_{i\to\infty}a_i$ and almost sure convergence is denoted $\text{as-lim}_{i\to\infty}a_i$ \cite[Definition 7.2.1]{grimmett2001probability}.

\section{Inspiration for LTV Watermarking}\label{sec:inspire}
This section describes the inspiration for LTV dynamic watermarking by summarizing the method described in Hespanhol et al. \cite{Hespanhol2017} for LTI systems.
Consider an LTI system with state $x_n$, measurement $y_n$, process noise $w_n$, measurement noise $z_n$, watermark $e_n$, additive attack $v_n$, and stabilizing feedback that uses the observed state $\hat{x}$
\begin{align}
    x_{n+1}&=Ax_n+BK\hat{x}_n+Be_n+w_n\label{eq:LTI_state_update}\\
    \hat{x}_{n+1}&=(A+BK+LC)\hat{x}_n+Be_n-Ly_n\label{eq:LTI_observer_update}\\
    y_n&=Cx_n+z_n+v_n\label{eq:LTI_measurment}
\end{align}
where $x_n,\hat{x}_n,w_n\in \mathbb{R}^p$, $e_n\in\mathbb{R}^q$, $y_n,z_n,v_n\in\mathbb{R}^r$, and $x_0=0_{p\times 1}$.
The process noise $w_n$, measurement noise $z_n$, and watermark $e_n$ are mutually independent and take the form $w_n\sim\mathcal{N}(0_{p\times 1},\Sigma_{w})$, $z_n\sim\mathcal{N}(0_{r\times 1},\Sigma_{z})$, and $e_n\sim\mathcal{N}(0_{q\times 1},\Sigma_e)$. 
While the process and measurement noise are unknown to the controller, the watermark signal is generated by the controller and is known. 
The following assumption is made on the controller, observer, and watermark design.
\begin{asmp}
Assume $\|A+BK\|<1$, $\|A+LC\|<1$, and $\Sigma_e$ is full rank.
\end{asmp}

The measurement residual for this system takes the form $C\hat{x}_n-y_n$.
When an attack is not present, the distribution of the measurement residuals converge to a zero mean Gaussian distribution with covariance $\Sigma$ where
\begin{align}
    \Sigma=\lim_{n\to\infty}~\mathds{E}[(C\hat{x}_n-y_n)(C\hat{x}_n-y_n)^\T].\label{eq:LTI_expected_value}
\end{align}
Note, for a LTV system, the limit in \eqref{eq:LTI_expected_value} may not exist.

Next, consider a generalization of a replay attack satisfying
\begin{align}
    v_n&=\alpha(Cx_n+z_n)+C\xi_n+\zeta_n\label{eq:LTI_attack}\\
    \xi_{n+1}&=(A+BK)\xi_n+\omega_n\label{eq:LTI_false_state_update}
\end{align}
where $\alpha\in\mathbb{R}$ is called the \emph{attack scaling factor}, the false state $\xi_n\in\mathbb{R}^p$ has process noise $\omega_n\in\mathbb{R}^p$ and measurement noise $\zeta_n\in\mathbb{R}^r$ that take the form $\omega_n\sim\mathcal{N}(0_{p\times 1},\Sigma_\omega)$ and $\zeta_n\sim\mathcal{N}(0_{r\times 1},\Sigma_\zeta)$, and are mutually independent with each other and with $w_n$ and $z_n$.
Note, when $\Sigma_\omega$ and $\Sigma_\zeta$ are selected such that the covariance of the measurement residual is unaltered and the attack scaling parameter is $-1$, this model describes a replay attack.
While attackers may have the ability to start and stop attacks at will, attacks that are only present for finite time are not guaranteed to be detected.
Therefore, when considering asymptotic guarantees of detection, the assumption of persistence is made.
To formally describe these persistent attacks, consider the following definition.
\begin{df} \label{def:LTI_power}
The \underline{asymptotic attack power} is defined as
\begin{align}
    \underset{i\to\infty}{\textup{as-lim}}~\textstyle\frac{1}{i}\sum_{n=0}^{i-1}v_n^\T v_n.\label{eq:LTI_attack_power}
\end{align}
\end{df}
\noindent Under this definition, an attack with non-zero asymptotic power is deemed to be persistent.

The asymptotic claims of LTI dynamic watermarking take the form of the following theorem.
\begin{theorem}\cite[Theorem 1]{Hespanhol2017}\label{thm:LTI_asymptotic_tests}
Consider an attacked LTI system satisfying \eqref{eq:LTI_state_update}-\eqref{eq:LTI_measurment}, an attack model satisfying \eqref{eq:LTI_attack}-\eqref{eq:LTI_false_state_update}, and $\Sigma$ satisfying \eqref{eq:LTI_expected_value}.
Let $k^\prime=\min\{k\geq0~|~C(A+BK)^kB\neq0_{r\times q}\}$ be finite.
If 
\begin{align}
    \underset{i\to \infty}{\textup{as-lim}}&~\textstyle\frac{1}{i}\sum_{n=0}^{i-1}(C\hat{x}_n-y_n)(C\hat{x}_n-y_n)^\T=\Sigma\label{eq:LTI_test1}\\
    \intertext{and}
    \underset{i\to \infty}{\textup{as-lim}}&~\textstyle\frac{1}{i}\sum_{n=0}^{i-1}(C\hat{x}_n-y_n)e_{n-k^\prime-1}^\T=0_{r\times q},\label{eq:LTI_test2}
\end{align}
then the asymptotic attack power is 0.
\end{theorem}
The delay of the watermark by $k^\prime$ in \eqref{eq:LTI_test2} ensures that the effect of the watermark is present in the measurement signal.
Note, the contrapositive of Theorem \ref{thm:LTI_asymptotic_tests} states that for attacks with non-zero asymptotic power, \eqref{eq:LTI_test1} and \eqref{eq:LTI_test2} cannot both be satisfied.
Therefore, considering the LHS of \eqref{eq:LTI_test1} and \eqref{eq:LTI_test2}, generalized replay attacks of non-zero asymptotic power are guaranteed to be detected in infinite time.

To make these tests implementable in real time, a statistical test is derived using a sliding window of fixed size.
At each step, the combined partial sums in \eqref{eq:LTI_test1}-\eqref{eq:LTI_test2} take the form
\begin{align}
    S_n=\textstyle\sum_{i=n+1}^{n+\ell}\begin{bmatrix}(C\hat{x}_i-y_i)\\e_{i-k^\prime-1}\end{bmatrix}\begin{bmatrix}(C\hat{x}_i-y_i)^\T & e_{i-k^\prime-1}^\T\end{bmatrix}.\label{eq:LTI_test_matrix}
\end{align}
Under the assumption of no attack, $S_n$ converges asymptotically to the Wishart distribution with scale matrix 
\begin{align}
     S=\begin{bmatrix}\Sigma & 0_{r\times q}\\0_{q\times r} & \Sigma_e\end{bmatrix}
\end{align}
and $\ell$ degrees of freedom as $\ell\to\infty$. 
Furthermore, for a generalized replay attack of non-zero asymptotic power, Theorem \ref{thm:LTI_asymptotic_tests} gives us that the scale matrix for $S_n$ is no longer $S$, since either \eqref{eq:LTI_test1} or \eqref{eq:LTI_test2} is not satisfied.
Given the sampled matrix $S_n$, the test then uses the negative log likelihood of the scale matrix
\begin{align}
    \mathcal{L}(S_n)=(m+q+1-\ell)\log(|S_n|)+\text{tr}\left(S^{-1}S_n\right).
\end{align}
Negative log likelihood values that exceed a user defined threshold, signal an attack.

For LTV systems, the limits in \eqref{eq:LTI_test1}-\eqref{eq:LTI_test2} may not exist. 
Furthermore, the sampled matrices $S_n$ may no longer be approximated as a Wishart distribution since the vectors used to create it in \eqref{eq:LTI_test_matrix} are not necessarily identically distributed. 
To accommodate these changes in distribution, it is necessary to develop a new method.




\section{LTV Dynamic Watermarking}\label{sec:LTV_theory}
This section derives the limit-based formulation of Dynamic Watermarking for a discrete-time LTV system. 
First, the LTV dynamics, necessary assumptions, and the resulting limit based tests are defined. 
Subsequently, Section \ref{sec:intermediate_results} provides intermediate results to prove these claims.

Consider an LTV system with state $x_n$, measurement $y_n$, process noise $w_n$, measurement noise $z_n$, watermark $e_n$, additive attack $v_n$, and stabilizing feedback that uses the observed state $\hat{x}$
\begin{align}
    x_{n+1}&=A_nx_n+B_nK_n\hat{x}_n+B_ne_n+w_n\label{eq:state_update}\\
    y_n&=C_nx_n+z_n+v_n\label{eq:output_equation}
\end{align}
where $x_n,\hat{x}_n,w_n\in \mathbb{R}^p$, $e_n\in\mathbb{R}^q$, $y_n,z_n,v_n\in\mathbb{R}^r$, and $x_0=0_{p\times 1}$.
The process noise $w_n$, measurement noise $z_n$, and watermark $e_n$ are mutually independent and take the form $w_n\sim\mathcal{N}(0_{p\times 1},\Sigma_{w,n})$, $z_n\sim\mathcal{N}(0_{r\times 1},\Sigma_{z,n})$, and $e_n\sim\mathcal{N}(0_{q\times 1},\Sigma_e)$. 
While the process and measurement noise are unknown to the controller, the watermark signal is generated by the controller and is known.
For simplicity, define $\AK_n=(A_n+B_nK_n)$ and $\AK_{(n,m)}=\AK_n\cdots\AK_m$ for $n\geq m$ and $\AK_{(n,n+1)}=I_p$.
We make the following assumption.
\begin{asmp}\label{asmp:bound1}
The covariances $\Sigma_e$, $\Sigma_{w,n}$, and $\Sigma_{z,n}$, of the random variables used in \eqref{eq:state_update}-\eqref{eq:output_equation}, are full rank.
Furthermore, there exists positive constants $\eta_w,\eta_z,\eta_{\AK},\eta_B,\eta_C\in\mathbb{R}$ such that $\|\Sigma_{w,n}\|<\eta_w$, $\|\Sigma_{z,n}\|<\eta_z$, $\|\AK_n\|<\eta_{\AK}<1$, $\|B_n\|<\eta_B$, and $\|C_n\|<\eta_C$, for all $n\in\mathbb{N}$.
\end{asmp}
\noindent The assumption of bounded full rank covariances for the process and measurement noise are satisfied for most systems by modeling error and sensor noise.
Furthermore, the input and output matrices are often constrained to be finite by sensor and actuator limits.
Since the watermark and controller are user defined, the remainder of the assumptions can be satisfied so long as the system is controllable.
We make the following assumption.
\begin{asmp}
\begin{align}
    \lim_{i\to\infty}~\textstyle\frac{1}{i}\sum_{n=0}^{i-1} C_nB_{n-1}\neq 0_{r\times q}.\label{eq:consistently_observervable_condition}
\end{align}
\end{asmp}
\noindent Here, \eqref{eq:consistently_observervable_condition} guarantees an asymptotic correlation between the measurement signal $y_n$ and the watermark signal $e_{n-1}$, which has been delayed by a single time step.
This ensures that the watermark has a persistent measurable effect on the measurement signal, which can then be used for validation purposes.
This is similar to assuming $k^\prime$ is equal to 0 for the LTI case.

The observer and the corresponding observer error, defined as $\delta_n=\hat{x}_n-x_n$, satisfy 
\begin{align}
    \hat{x}_{n+1}&=(\AK_n+L_nC_n)\hat{x}_n+B_ne_n-L_ny_n\label{eq:obsver_update}\\
    \delta_{n+1}&=(A_n+L_nC_n)\delta_n-w_n-L_n(z_n+v_n),\label{eq:observer_error_update_full}
\end{align}
where $\hat{x}_0=\delta_0=0_{p\times 1}$. 
For simplicity, define $\AL_n=(A_n+L_nC_n)$ and $\AL_{(n,m)}=\AL_n\cdots\AL_m$ for $n\geq m$ and $\AL_{(n,n+1)}=I_p$.
Furthermore, let
\begin{align}
    \bar{\delta}_{n+1}&=\AL_n\bar{\delta}_n-w_n-L_nz_n\label{eq:observer_error_update_normal}\\
    \delAttack_{n+1}&=\AL_n\delAttack_{n}-L_nv_n\label{eq:observer_error_update_attack}
\end{align}
where $\bar{\delta}_0=\delAttack_0=0_{p\times1}$. 
Note that $\delta_n=\bar{\delta}_n+\delAttack_n$ and that when $v_n=0_{r\times1}, ~\forall n$ we have that $\delAttack_n=0_{p\times1}, ~\forall n$. 
Here $\bar{\delta}_n$ can be thought of as the portion of the observer error that results from the original noise of the system, while $\delAttack_n$ is the contribution of the attack to the observer error. 

Next, consider the expected value $\Sigma_{\delta,n}=\mathds{E}[\bar{\delta}_n\bar{\delta}_n^\T]=\mathds{E}[\delta_n\delta_n^\T~|~v_n=0_{r\times 1},~\forall n]$, which can be written as
\begin{align}
    \Sigma_{\delta,n}&=\textstyle\sum _{i=0}^n \AL_{(n-1,n-i+1)}(\Sigma_{w,n-i}+\nonumber\\
    &\qquad+L_{n-i}\Sigma_{z,n-i}L_{n-i}^\T)\AL_{(n-1,n-i+1)}^\T.\label{eq:observer_error_covariance_normal}
\end{align}
The \emph{matrix normalization factor} is then defined as
\begin{align}
    V_n=(C_n\Sigma_{\delta,n}C_n^\T+\Sigma_{z,n})^{-1/2}\label{eq:residual_normalizer},
\end{align}
which exists since $\Sigma_{z,n}$ is full rank. 
For an LTI system, the matrix $V_n=\Sigma^{-1/2}$ where $\Sigma$ is as defined in \eqref{eq:LTI_expected_value}.
For the LTV system, the matrix normalization factor can be thought of as a time-varying normalization for the measurement residual.
Next, we make the following assumption about the observer.
\begin{asmp}\label{asmp:obs_bounds}
    There exists positive constants $\eta_{\AL},$ $\eta_L,$ $\eta_\delta,$
    $\eta_V\in\mathbb{R}$ such that $\|\AL_n\|<\eta_{\AL}<1$, $\|L_n\|<\eta_L$, $\|\Sigma_{\delta,n}\|<\eta_\delta$, and $\|V_n\|<\eta_V$, for all $n\in\mathbb{N}$.
\end{asmp}
\noindent If the system in \eqref{eq:state_update}-\eqref{eq:output_equation} is observable, then the user defined controller can satisfy the assumption on $\AL_n$.  
Previous assumptions imply the assumptions on $L_n,~\Sigma_{\delta,n},$ and $V_n$ are satisfied, but the bounds here are used to simplify notation.

Next, we alter the attack defined in \eqref{eq:LTI_attack}-\eqref{eq:LTI_false_state_update} to create a time-varying equivalent.
Consider an attack $v_n$ that satisfies
\begin{align}
    v_n&=\alpha(C_nx_n+z_n)+C_n\xi_n+\zeta_n\label{eq:attack_definition}\\
    \xi_{n+1}&=\AK_n\xi_n+\omega_n\label{eq:false_state_update},
\end{align}
where $\alpha\in\mathbb{R}$ is called the \emph{attack scaling factor}, the \emph{false state} $\xi_n\in\mathbb{R}^p$ has process noise $\omega_n\in\mathbb{R}^p$ and measurement noise $\zeta_n\in\mathbb{R}^r$ that take the form $\omega_n\sim\mathcal{N}(0_{p\times 1},\Sigma_{\omega,n})$ and  $\zeta_n\sim\mathcal{N}(0_{r\times 1},\Sigma_{\zeta,n})$ and are mutually independent with each other and with $w_n$ and $z_n$. 
Similar to the LTI case, when $\Sigma_{\omega,n}$ and $\Sigma_{\zeta,n}$ are selected properly and the attack scaling parameter is $-1$, this model describes a replay attack.
The results of such an attack can have devastating results as shown in Figure \ref{fig:LTV_attacked}.
While an attacker could choose to allow the noise to have unbounded covariance, the resulting attack would be trivial to detect.
Therefore, we make the following assumption about the attack model.
\begin{asmp}
When there is an attack, $v_n$ follows the dynamics \eqref{eq:attack_definition}-\eqref{eq:false_state_update} with the attack scaling factor remaining constant.
Furthermore, there exists positive constants $\eta_\omega,\eta_\eta\in\mathbb{R}$ such that $\|\Sigma_{\omega,n}\|<\eta_\omega,~\|\Sigma_{\zeta,n}\|<\eta_\zeta$, for all $n\in\mathbb{N}$.
\end{asmp}

To make asymptotic guarantees of detection, we also assume the persistence of attacks using the following definition. 
\begin{df}\label{def:LTV_power}
The \underline{asymptotic attack power} is defined as
\begin{align}
    \plim{i}~\textstyle\frac{1}{i}\sum_{n=0}^{i-1}v_n^\T v_n.\label{eq:attack_power}
\end{align}
\end{df}

Similar to prior research in Dynamic Watermarking, we first define the asymptotic tests.
\begin{theorem}\label{thm:LTV_Asymptotic_Main_Result}
Consider an attacked LTV system satisfying the dynamics in \eqref{eq:state_update}-\eqref{eq:observer_error_update_attack}. 
Let $V_n$ be as defined in \eqref{eq:residual_normalizer}.
If $v_n=0_{r\times 1}$, for all $n\in\mathbb{N}$, then 
\begin{align}
    \plim{i}~&\textstyle\frac{1}{i}\sum_{n=0}^{i-1} V_n(C_n\hat{x}_n-y_n)e_{n-1}^\T=0_{r\times q} \tag{C1} \label{eq:ltv_watermark_correlation_test}\\
    \intertext{and}
    \plim{i}~&\textstyle\frac{1}{i}\sum_{n=0}^{i-1} V_n(C_n\hat{x}_n-y_n)(C_n\hat{x}_n-y_n)^\T V_n^\T=I_r \tag{C2} \label{eq:ltv_covariance_test}.
\end{align}
Furthermore, if the attack follows the dynamics in \eqref{eq:attack_definition}-\eqref{eq:false_state_update} and has non-zero asymptotic power as defined in Definition \ref{def:LTV_power}, then \eqref{eq:ltv_watermark_correlation_test} and \eqref{eq:ltv_covariance_test} cannot both be satisfied.
\end{theorem}
From Theorem \ref{thm:LTV_Asymptotic_Main_Result}, the LHS of \eqref{eq:ltv_watermark_correlation_test} and \eqref{eq:ltv_covariance_test} can be used to guarantee detection of generalized replay attacks with non-zero asymptotic power in infinite time.
Note, \eqref{eq:ltv_watermark_correlation_test}, \eqref{eq:ltv_covariance_test}, and \eqref{eq:attack_power} use limits in probability as opposed to the almost sure limits used in their LTI counterparts. 
This change removes the guarantee of detection via the asymptotic tests for certain pathological examples of attacks, but both forms of convergence provide the same motivation for the statistical tests in Section \ref{sec:LTV_practical}.
Given an arbitrary real number $\epsilon$, almost sure convergence states that with probability 1 the sequence will remain a distance of less than $\epsilon$ from the limit after a finite number of steps while convergence in probability states that the probability that an element of the sequence is within a distance of $\epsilon$ from the limit converges to 1 as you continue along the sequence.
Since the statistical tests use a sliding window to consider only a finite number of steps at a time, both forms of convergence say that as the window size grows the sequence of sample averages become more likely to be closer to the limit when no attack is present. 
As a result, the test becomes more sensitive.

\subsection{Intermediate Results}\label{sec:intermediate_results}
To prove Theorem \ref{thm:LTV_Asymptotic_Main_Result}, several intermediate results must first be provided.
First, we consider the asymptotic limit \eqref{eq:ltv_watermark_correlation_test} and show that it implies $\alpha$ is equal to 0.
This allows us to assume that $\alpha$ is equal to 0 for the remainder of the intermediate results.
\begin{theorem}\label{thm:alpha} 
Consider an attacked LTV system satisfying \eqref{eq:state_update}-\eqref{eq:observer_error_update_attack} and the attack model satisfying \eqref{eq:attack_definition}-\eqref{eq:false_state_update}.
Let $V_n$ be as defined in \eqref{eq:residual_normalizer}. \eqref{eq:ltv_watermark_correlation_test} holds if and only if the attack scaling factor $\alpha$ is equal to 0.
\end{theorem}
\begin{proof} (Theorem \ref{thm:alpha})
Assume that $\alpha$ is equal to 0.
Rearranging the LHS of (\ref{eq:ltv_watermark_correlation_test}) using \eqref{eq:output_equation}, \eqref{eq:observer_error_update_full}, and \eqref{eq:attack_definition} results in 
\begin{align}
    &\plim{i}~\textstyle\frac{1}{i}\sum_{n=0}^{i-1} V_n(C_n\hat{x}_n-y_n)e_{n-1}^\T=\nonumber\\
    &~=\plim{i}~\textstyle\frac{1}{i}\sum_{n=0}^{i-1} V_n(C_n\delta_n-z_n-C_n\xi_n-\zeta_n)e_{n-1}^\T.\label{eq:alpha_C1_first_expand_a0}
\end{align}
Corollary \ref{cor:combine_prob_limit} says that to show that the RHS of \eqref{eq:alpha_C1_first_expand_a0} converges in probability to $0_{r\times q}$, it is sufficient to show that each term in the sum converges in probability to $0_{r\times q}$. 
Note that 
\begin{align}
    &\plim{i}~\textstyle\frac{1}{i}\sum_{n=0}^{i-1} V_n(C_n\delta_n-C_n\xi_n)e_{n-1}^\T
    =0_{r\times q}
\end{align}
by Corollary \ref{cor:converge_expected_cor} since $e_{n-1}$ is independent identically distributed with bounded covariance, and $V_n(C_n\delta_n-C_n\xi_n)$ is a bounded linear transform of a random vector that satisfies the necessary auto correlation bound as a result of Theorem \ref{thm:alpha_zero_bounded_state}. 
Similarly,
\begin{align}
    \plim{i}~\textstyle\frac{1}{i}\sum_{n=0}^{i-1} V_n(-z_n-\zeta_n)e_{n-1}^\T=0_{r\times q}\label{eq:alpha_measure_noise_cancel}
\end{align}
by Corollary \ref{cor:converge_expected_cor} since $z_n$, $\zeta_n$, and $e_{n-1}$ are mutually independent identically distributed with bounded covariances. 
Therefore $\alpha=0$ implies \eqref{eq:ltv_watermark_correlation_test} holds.

Now assume that \eqref{eq:ltv_watermark_correlation_test} holds. 
Rearranging (\ref{eq:ltv_watermark_correlation_test}) using \eqref{eq:output_equation}, \eqref{eq:observer_error_update_full}, and \eqref{eq:attack_definition} results in 
\begin{align}
    \plim{i}~&\textstyle\frac{1}{i}\sum_{n=0}^{i-1} V_n(C_n\hat{x}_n-y_n)e_{n-1}^\T=\nonumber\\
    &=\plim{i}~\textstyle\frac{1}{i}\sum_{n=0}^{i-1} V_n(C_n\delta_n-(1+\alpha)z_n+\nonumber\\
    &\hspace{0.8in}-\alpha C_nx_n-C_n\xi_n-\zeta_n)e_{n-1}^\T.\label{eq:alpha_C1_first_expand}
\end{align}
Now since \eqref{eq:alpha_measure_noise_cancel} holds by the same argument as before, we can use Theorem \ref{thm:split_prob_limit} to cancel these terms resulting in 
\begin{align}
    \plim{i}~&\textstyle\frac{1}{i}\sum_{n=0}^{i-1} V_n(C_n\hat{x}_n-y_n)e_{n-1}^\T=\nonumber\\
    &=\plim{i}~\textstyle\frac{1}{i}\sum_{n=0}^{i-1} V_n(C_n\delta_n-\alpha C_nx_n)e_{n-1}^T.\label{eq:alpha_C1_first_cancel}
\end{align}
Expanding $x_n$ in \eqref{eq:alpha_C1_first_cancel} by one step using \eqref{eq:state_update} then results in 
\begin{align}
    &\plim{i}~\textstyle\frac{1}{i}\sum_{n=0}^{i-1} V_n(C_n\hat{x}_n-y_n)e_{n-1}^\T=\nonumber\\
    &~=\plim{i}~\textstyle\frac{1}{i}\sum_{n=0}^{i-1} V_n(C_n\delta_n-\alpha C_n(A_{n-1}x_{n-1}+\nonumber\\
    &~\quad+B_{n-1}K_{n-1}\hat{x}_{n-1}+B_{n-1}e_{n-1}+w_{n-1}))e_{n-1}^\T.\label{eq:alpha_C1_second_expand}
\end{align}
Using Corollary \ref{cor:converge_expected_cor} we have that 
\begin{align}
    \plim{i}~&\textstyle\frac{1}{i}\sum_{n=0}^{i-1}-\alpha V_nC_nB_{n-1}(e_{n-1}e_{n-1}^\T-\Sigma_e)=0_{q\times r}.
\end{align}
Therefore by Theorem \ref{thm:split_prob_limit} we have
\begin{align}
    &\plim{i}~\textstyle\frac{1}{i}\sum_{n=0}^{i-1} V_n(C_n\hat{x}_n-y_n)e_{n-1}^\T=\nonumber\\
    &~=\plim{i}~\textstyle\frac{1}{i}\sum_{n=0}^{i-1} V_n(C_n\delta_n-\alpha C_n(A_{n-1}x_{n-1}+\nonumber\\
    &~\quad+B_{n-1}K_{n-1}\hat{x}_{n-1}+w_{n-1}))e_{n-1}^\T+\alpha V_nC_nB_{n-1}\Sigma_e.\label{eq:alpha_C1_second_cancel}
\end{align}
Note, that all elements of
\begin{align}
    V_n(C_n\delta_n-\alpha C_n(&A_{n-1}x_{n-1}+\nonumber\\
    &+B_{n-1}K_{n-1}\hat{x}_{n-1}+w_{n-1}))e_{n-1}^\T
\end{align}
are distributed symmetrically about 0 for all $n\in\mathbb{N}$.
Consider an element of \eqref{eq:alpha_C1_second_cancel} for which the corresponding element in
\begin{align}\label{eq:C1_non_converge}
    \textstyle\frac{1}{i}\sum_{n=0}^{i-1} V_nC_nB_{n-1}\Sigma_e
\end{align}
does not converge.
For each $i$, the probability that the matrix element in \eqref{eq:alpha_C1_second_cancel} is farther away from 0 than the corresponding element in \eqref{eq:C1_non_converge} is at least $0.5$.
Therefore the element cannot converge in probability to 0 completing the proof.
\end{proof}

Assuming $\alpha$ is equal to 0, we show that \eqref{eq:ltv_covariance_test} is equivalent to another condition that is only dependent on the attack $v_n$ and its contribution to the observer error $\delAttack_n$. 
Note, $\delAttack_n$ is not a computable quantity given the available knowledge of the system, but the provided intermediate condition is an amenable surrogate to \eqref{eq:ltv_covariance_test}.
\begin{theorem}\label{thm:equiv_attack_covariance}
Consider an attacked LTV system satisfying \eqref{eq:state_update}-\eqref{eq:observer_error_update_attack} and an attack model satisfying \eqref{eq:attack_definition}-\eqref{eq:false_state_update}.
Let $V_n$ be as defined in \eqref{eq:residual_normalizer}.
Assume the attack scaling factor $\alpha$ is equal to 0.
\eqref{eq:ltv_covariance_test} holds if and only if 
\begin{align}
    \plim{i}~\textstyle\frac{1}{i}\sum_{n=0}^{i-1} V_n(C_n\delAttack_n-v_n)(C_n\delAttack_n-v_n)^\T V_n^\T =0_r.\label{eq:attack_additive_covariance}
\end{align}
\end{theorem}
\begin{proof} (Theorem \ref{thm:equiv_attack_covariance})
Expanding \eqref{eq:ltv_covariance_test} using \eqref{eq:output_equation} and \eqref{eq:observer_error_update_full}-\eqref{eq:observer_error_update_attack} gives us 
\begin{align}
    \plim{i}~&\textstyle\frac{1}{i}\sum_{n=0}^{i-1} V_n(C_n\hat{x}_n-y_n)(C_n\hat{x}_n-y_n)^\T V_n^\T=\nonumber\\
    &=\plim{i}~\textstyle\frac{1}{i}\sum_{n=0}^{i-1}V_n(C_n\bar{\delta}_n-z_n)(C_n\bar{\delta}_n-z_n)^\T V_n^\T+\nonumber\\
    &\qquad+V_n(C_n\bar{\delta}_n-z_n)(C_n\delAttack_n-v_n)^\T V_n^\T+\nonumber\\
    &\qquad+V_n(C_n\delAttack_n-v_n)(C_n\bar{\delta}_n-z_n)^\T V_n^\T+\nonumber\\
    &\qquad+V_n(C_n\delAttack_n-v_n)(C_n\delAttack_n-v_n)^\T V_n^\T.\label{eq:equiv_attack_covariance_expand}
\end{align}
By Corollary \ref{cor:converge_expected_cor} and Theorem \ref{thm:alpha_zero_bounded_state},
\begin{align}
    \plim{i}~&\textstyle\frac{1}{i}\sum_{n=0}^{i-1} V_n(C_n\bar{\delta}_n-z_n)(C_n\bar{\delta}_n-z_n)^\T V_n^\T=I_r\label{eq:equiv_attack_covariance_converge_I}
\end{align}
and
\begin{align}
    \plim{i}~&\textstyle\frac{1}{i}\sum_{n=0}^{i-1} V_n(C_n\bar{\delta}_n-z_n)(C_n\delAttack-v_n)^\T V_n^\T=0_{r}\label{eq:equiv_attack_covariance_cross}
\end{align}
since, by the definition of $V_n$ in \eqref{eq:residual_normalizer}, the expectation for each summand in \eqref{eq:equiv_attack_covariance_converge_I} is $I_r$, and $V_n(C_n\bar{\delta}_n-z_n)$ is uncorrelated with $V_n(C_n\delAttack-v_n)$.
First, assume that \eqref{eq:ltv_covariance_test} holds.
By Theorem \ref{thm:split_prob_limit}, it follows from \eqref{eq:equiv_attack_covariance_expand}-\eqref{eq:equiv_attack_covariance_cross} that \eqref{eq:attack_additive_covariance} must hold. 
Next, assume that \eqref{eq:attack_additive_covariance} holds.
By Corollary \ref{cor:combine_prob_limit}, it follows from \eqref{eq:equiv_attack_covariance_expand}-\eqref{eq:equiv_attack_covariance_cross} that \eqref{eq:ltv_covariance_test} holds.
\end{proof}

Since the attack $v_n$, under the assumption that $\alpha$ is equal to 0, is only dependent on the random vectors $\xi_n$ and $\zeta_n$, we now provide sufficient and necessary conditions on these random vectors for the asymptotic attack power to be 0.
Similar to Theorem \ref{thm:equiv_attack_covariance}, these random vectors are not computable by the controller, but the resulting conditions can be used to connect \eqref{eq:ltv_covariance_test} to the asymptotic attack power.
\begin{theorem}\label{thm:att_equiv} 
Consider an attacked LTV system satisfying \eqref{eq:state_update}-\eqref{eq:observer_error_update_attack} and an attack model satisfying \eqref{eq:attack_definition}-\eqref{eq:false_state_update}.
Assume that the attack scaling factor $\alpha$ is equal to 0.
The asymptotic attack power as defined in \eqref{eq:attack_power} is 0 if and only if
\begin{align}
    \plim{i}\textstyle\frac{1}{i}\sum_{n=0}^{i-1}\zeta_n\zeta_n^\T=0\label{eq:att_equiv_attack_noise_claim}
\end{align}
and
\begin{align}
    \plim{i}\textstyle\frac{1}{i}\sum_{n=0}^{i-1}C_n\xi_n\xi_n^\T C_n^\T=0_r\label{eq:att_equiv_attack_state_claim}.
\end{align}
\end{theorem}
\begin{proof}(Theorem \ref{thm:att_equiv})
Assume that $\alpha=0$. 
Using Lemma \ref{lm:mat_scal} we have that the asymptotic attack power is 0 if and only if
\begin{align}
    \plim{i}&\textstyle\frac{1}{i}\sum_{n=0}^{i-1}v_nv_n^\T=0_{r}.\label{eq:attack_equiv_power_mat}
\end{align}
Expanding the LHS of \eqref{eq:attack_equiv_power_mat} using \eqref{eq:attack_definition}-\eqref{eq:false_state_update} we get an equvalent condition. 
\begin{align}
    &\plim{i}\textstyle\frac{1}{i}\sum_{n=0}^{i-1}C_n\xi_n\xi_n^\T C_n^\T+\nonumber\\
    &\quad+C_n\xi_n\zeta_n^\T+(C_n\xi_n\zeta_n^\T)^\T+\zeta_n\zeta_n^\T=0_r\label{eq:att_equiv_expand}
\end{align}
Since $\xi_n$ and $\zeta_n$ are uncorrelated, from Theorem \ref{thm:alpha_zero_bounded_state} and Corollary \ref{cor:converge_expected_cor} we have
\begin{align}
    \plim{i}&\textstyle\frac{1}{i}\sum_{n=0}^{i-1} C_n\xi_n\zeta_n^\T=0_r.\label{eq:att_equiv_cross_cancel}
\end{align}
First, assume that \eqref{eq:att_equiv_attack_noise_claim} and \eqref{eq:att_equiv_attack_state_claim} hold. 
By Corollary \ref{cor:combine_prob_limit} we have that \eqref{eq:att_equiv_expand} must hold since, when separated, the limit for each term converges to $0_r$. 
Next, assume that \eqref{eq:att_equiv_expand} holds.
By Theorem \ref{thm:split_prob_limit} we can rewrite \eqref{eq:att_equiv_expand} as 
\begin{align}
    \plim{i}\textstyle\frac{1}{i}\sum_{n=0}^{i-1}\zeta_n\zeta_n^\T+C_n\xi_n\xi_n^\T C_n^\T=0_r,\label{eq:att_equiv_canceled_cross}
\end{align}
since \eqref{eq:att_equiv_cross_cancel} holds.
Note, both terms are positive-semidefinite matrices.
Therefore, for an arbitrary $\epsilon>0$ we have that 
\begin{align}
    \mathds{P}&\left(\left\|\textstyle\frac{1}{i}\sum_{n=0}^{i-1}\zeta_n\zeta_n^\T\right\|>\epsilon\right)\leq\nonumber\\
    &\qquad\leq\mathds{P}\left(\left\|\textstyle\frac{1}{i}\sum_{n=0}^{i-1}\zeta_n\zeta_n^\T+C_n\xi_n\xi_n^\T C_n^\T\right\|>\epsilon\right)\label{eq:inclusion_bound}
\end{align}
Furthermore, \eqref{eq:att_equiv_canceled_cross} implies
\begin{align}
    \lim_{i\to\infty}\mathds{P}\left(\left\|\textstyle\frac{1}{i}\sum_{n=0}^{i-1}\zeta_n\zeta_n^\T+C_n\xi_n\xi_n^\T C_n^\T\right\|>\epsilon\right)=0_r,~\forall\epsilon>0\label{eq:whole_prob}
\end{align}
Then, by \eqref{eq:inclusion_bound} and \eqref{eq:whole_prob}
\begin{align}
    \lim_{i\to\infty}\mathds{P}\left(\left\|\textstyle\frac{1}{i}\sum_{n=0}^{i-1}\zeta_n\zeta_n^\T\right\|>\epsilon\right)=0_r,~\forall\epsilon>0.
\end{align}
Therefore, \eqref{eq:att_equiv_attack_noise_claim} must hold.
Applying Theorem \ref{thm:split_prob_limit} to \eqref{eq:att_equiv_canceled_cross} using \eqref{eq:att_equiv_attack_noise_claim} implies \eqref{eq:att_equiv_attack_state_claim} must also hold.
\end{proof}

Next, we start to complete the connection between \eqref{eq:ltv_covariance_test} and zero asymptotic attack power by proving \eqref{eq:attack_additive_covariance} implies \eqref{eq:att_equiv_attack_noise_claim}.
Furthermore, we prove a related result that makes it simpler to prove that \eqref{eq:attack_additive_covariance} implies \eqref{eq:att_equiv_attack_state_claim}. 
\begin{theorem}\label{thm:me_noise}
Consider an attacked LTV system satisfying \eqref{eq:state_update}-\eqref{eq:observer_error_update_attack} and an attack model satisfying \eqref{eq:attack_definition}-\eqref{eq:false_state_update}.
Let $V_n$ be as defined in \eqref{eq:residual_normalizer}.
Assume the attack scaling factor $\alpha$ is equal to 0.
If \eqref{eq:attack_additive_covariance} holds, then \eqref{eq:att_equiv_attack_noise_claim} holds as well and
\begin{align}
    \plim{i}\textstyle\frac{1}{i}\sum_{n=0}^{i-1}(C_n\delAttack_n-C_n\xi_n)(C_n\delAttack_n-C_n\xi_n)^\T=0_r.\label{eq:me_noise_state_noise_claim}
\end{align}
\end{theorem}
\begin{proof}(Theorem \ref{thm:me_noise})
Assume that \eqref{eq:attack_additive_covariance} holds. 
Expanding the LHS of \eqref{eq:attack_additive_covariance} using \eqref{eq:attack_definition} we get 
\begin{align}
    &\plim{i}\textstyle\frac{1}{i}\sum_{n=0}^{i-1}V_n(C_n\delAttack_n-C_n\xi_n)(C_n\delAttack_n-C_n\xi_n)^\T V_n^\T+\nonumber\\
    &~~+V_n(C_n\delAttack_n-C_n\xi_n)\zeta_n^\T V_n^\T+(V_n(C_n\delAttack_n-C_n\xi_n)\zeta_n^\T V_n^\T)^\T+\nonumber\\
    &~~+V_n\zeta_n\zeta_n^\T V_n^\T=0_r.\label{eq:expanded_attack_additive_covariance}
\end{align}
Using Corollary \ref{cor:converge_expected_cor} and Theorem \ref{thm:alpha_zero_bounded_state} we have
\begin{align}
    \plim{i}\textstyle\frac{1}{i}\sum_{n=0}^{i-1} V_n(C_n\delAttack_n-C_n\xi_n)\zeta_n^\T V_n^\T=0_r.
\end{align}
Therefore, by applying Theorem \ref{thm:split_prob_limit} to \eqref{eq:expanded_attack_additive_covariance} we have
\begin{align}
    \plim{i}\frac{1}{i}&\textstyle\sum_{n=0}^{i-1}V_n(C_n\delAttack_n-C_n\xi_n)(C_n\delAttack_n-C_n\xi_n)^\T V_n^\T+\nonumber\\
    &\qquad+V_n\zeta_n\zeta_n^\T V_n^\T=0_r.\label{eq:canceled_attack_additive_covariance}
\end{align}
Note, both terms are positive-semidefinite matrices.
Using the same method used on \eqref{eq:att_equiv_canceled_cross}, we then have 
\begin{align}
    \plim{i}\frac{1}{i}&\textstyle\sum_{n=0}^{i-1}V_n\zeta_n\zeta_n^\T V_n^\T=0_r\label{eq:me_noise_attack_measure_noise}
\end{align}
and
\begin{align}
    \plim{i}\textstyle\frac{1}{i}\sum_{n=0}^{i-1}V_n(C_n\delAttack_n-C_n\xi_n)(C_n\delAttack_n-C_n\xi_n)^\T V_n^\T=0_r.\label{eq:me_noise_attack_state_noise}
\end{align}
We complete the proof using Lemma \ref{lm:nonconv} but we must first provide lower bound on the eigenvalues of $V_n^\T V_n$. 
Let $\lambda_n$ denote the smallest eignenvalue of $V_n^\T V_n$, then $\lambda_n$ is lower bounded since 
\begin{align}
    \lambda_n&=\frac{1}{\|(V_n^\T V_n)^{-1}\|}=\frac{1}{\|C_n\Sigma_{\delta,n}C_n^\T+\Sigma_{z,n}\|}\geq\nonumber\\
    &\geq\frac{1}{\eta_C^2\eta_\delta+\eta_z}>0.\label{eq:eigen_lower_bnd}
\end{align}
If we assume that \eqref{eq:att_equiv_attack_noise_claim} does not hold then applying Lemma \ref{lm:nonconv} contradicts \eqref{eq:me_noise_attack_measure_noise}. Therefore \eqref{eq:att_equiv_attack_noise_claim} must hold. Similarly, assuming that \eqref{eq:me_noise_state_noise_claim} does not hold would result in a contradiction with \eqref{eq:me_noise_attack_state_noise}. Therefore \eqref{eq:me_noise_state_noise_claim} must also hold.
\end{proof}

Next we prove that \eqref{eq:attack_additive_covariance} implies  \eqref{eq:att_equiv_attack_state_claim} to complete the relation between \eqref{eq:ltv_covariance_test} and the asymptotic attack power.
\begin{theorem}\label{thm:st_noise}
Consider an attacked LTV system satisfying \eqref{eq:state_update}-\eqref{eq:observer_error_update_attack} and an attack model satisfying \eqref{eq:attack_definition}-\eqref{eq:false_state_update}.
Let $V_n$ be as defined in \eqref{eq:residual_normalizer}.
Assume the attack scaling factor $\alpha$ is equal to 0.
If \eqref{eq:attack_additive_covariance} holds then \eqref{eq:att_equiv_attack_state_claim} holds as well.
\end{theorem}

To prove Theorem \ref{thm:st_noise}, we instead prove the contrapositive statement by assuming that \eqref{eq:att_equiv_attack_state_claim} does not hold, proving that \eqref{eq:me_noise_state_noise_claim} does not hold either, then using Theorem \ref{thm:me_noise} we complete the proof.
To do this, we split the summation in \eqref{eq:att_equiv_attack_state_claim} according to the following lemma.
This split allows us to disregard the cross terms on the LHS of \eqref{eq:me_noise_state_noise_claim} and shows that the remaining terms do not converge in probability to $0_r$.

\begin{lm}\label{thm:min_m}
Consider an attacked LTV system satisfying \eqref{eq:state_update}-\eqref{eq:observer_error_update_attack} and an attack model satisfying \eqref{eq:attack_definition}-\eqref{eq:false_state_update}.
Let $V_n$ be as defined in \eqref{eq:residual_normalizer}.
Assume the attack scaling factor $\alpha$ is equal to 0. 
If \eqref{eq:att_equiv_attack_state_claim} does not hold then there exists $m\in\mathbb{N}$ for which 
\begin{align}
    &\plim{i}\textstyle\frac{1}{i}\sum_{n=0}^{i-1}\left(C_n\sum_{j=1}^{m_n}\AK_{(n-1,n-j+1)}\omega_{n-j}\right)\nonumber\times\\
    &\qquad\qquad\times\left(C_n\textstyle\sum_{j=1}^{m_n}\AK_{(n-1,n-j+1)}\omega_{n-j}\right)^\T\neq0_r\label{eq:min_m_claim}.
\end{align}
where $m_n=\min\{n,m\}$. 
Furthermore, there exists an $m^\prime\in\mathbb{N}$ such that $m^\prime\leq m$ and 
\begin{align}
    &\plim{i}\textstyle\frac{1}{i}\sum_{n=0}^{i-1}C_n\AK_{(n-1,n-j+1)}\omega_{n-j}\nonumber\times\\
    &\hspace{1in}\times\omega_{n-j}^\T\AK_{(n-1,n-j+1)}^\T C_n^\T\neq0_r\label{eq:min_m_claim_single}
\end{align}
for $j=m^\prime$ but not for $j<m^\prime$.
\end{lm}
\begin{proof}(Lemma \ref{thm:min_m})
First, we prove the existence of $m$. 
Assume that \eqref{eq:att_equiv_attack_state_claim} does not hold. 
Expanding the LHS of \eqref{eq:att_equiv_attack_state_claim} using \eqref{eq:false_state_update} results in 
\begin{align}
    \plim{i}\frac{1}{i}&\textstyle\sum_{n=0}^{i-1}\left(\textstyle\sum_{j=1}^{n}C_n\AK_{(n-1,n-j+1)}\omega_{n-j}\right)\times\nonumber\\
    &\times\left(\textstyle\sum_{j=1}^{n}C_n\AK_{(n-1,n-j+1)}\omega_{n-j}\right)^\T\neq 0_r.
\end{align}
Then, using Lemma \ref{lm:mat_scal} we have that
\begin{align}
    \plim{i}\textstyle\frac{1}{i}\sum_{n=0}^{i-1}\left\|\textstyle\sum_{j=1}^{n}C_n\AK_{(n-1,n-j+1)}\omega_{n-j}\right\|^2\neq 0.
\end{align}
Since \eqref{eq:att_equiv_attack_state_claim} does not hold there exists $\epsilon,\tau>0$ such that 
\begin{align}
    \mathds{P}\left(\textstyle\frac{1}{i}\sum_{n=0}^{i-1}\left\|\textstyle\sum_{j=1}^{n}C_n\AK_{(n-1,n-j+1)}\omega_{n-j}\right\|^2>\epsilon\right)>\tau\label{eq:min_m_total_prob}
\end{align}
for infinitely many $i$. 
We prove that there exists an $m$ such that for each $i$ that \eqref{eq:min_m_total_prob} holds we have 
\begin{align}
   \mathds{P}\left(\textstyle\frac{1}{i}\sum_{n=0}^{i-1}\left\|\textstyle\sum_{j=1}^{m_n}C_n\AK_{(n-1,n-j+1)}\omega_{n-j}\right\|^2>\frac{\epsilon}{6}\right)>\frac{\tau}{4}\label{eq:min_m_trunc_prob}
\end{align}
which is equivalent to \eqref{eq:min_m_claim} as a result of Lemma \ref{lm:mat_scal}. 
To make statements on the truncated sum, we start by finding the relationship between the probability in the LHS of \eqref{eq:min_m_total_prob} and the probability in the LHS of \eqref{eq:min_m_trunc_prob}. 
For each $i$ such that \eqref{eq:min_m_total_prob} holds, we apply triangle inequality to get 
\begin{align}
    \tau&<\mathds{P}\bigg(\textstyle\frac{1}{i}\sum_{n=0}^{i-1}\bigg(\left\|\textstyle\sum_{j=1}^{m_n}C_n\AK_{(n-1,n-j+1)}\omega_{n-j}\right\|+\nonumber\\
    &\quad+\left\|\textstyle\sum_{j=m_n+1}^{n}C_n\AK_{(n-1,n-j+1)}\omega_{n-j}\right\|\bigg)^2>\epsilon\bigg).
\end{align}
Further expanding and applying Theorem \ref{thm:split_prob_addition} result in
\begin{align}
    &\tau<\mathds{P}\left(\textstyle\frac{1}{i}\sum_{n=0}^{i-1}\left\|\textstyle\sum_{j=1}^{m_n}C_n\AK_{(n-1,n-j+1)}\omega_{n-j}\right\|^2>\frac{\epsilon}{3}\right)+\nonumber\\
    &~+\mathds{P}\bigg(\textstyle\frac{2}{i}\sum_{n=0}^{i-1}\left\|\textstyle\sum_{j=1}^{m_n}C_n\AK_{(n-1,n-j+1)}\omega_{n-j}\right\|\times\nonumber\\
    &~\times\left\|\textstyle\sum_{j=m_n+1}^{n}C_n\AK_{(n-1,n-j+1)}\omega_{n-j}\right\|>\textstyle\frac{\epsilon}{3}\bigg)+\nonumber\\
    &~+\mathds{P}\left(\textstyle\frac{1}{i}\sum_{n=0}^{i-1}\left\|\textstyle\sum_{j=m_n+1}^{n}C_n\AK_{(n-1,n-j+1)}\omega_{n-j}\right\|^2>\frac{\epsilon}{3}\right).\label{eq:min_m_full_expand}
\end{align}
Focusing on the center term in the RHS of \eqref{eq:min_m_full_expand}, we can write
\begin{align}
    &\mathds{P}\Bigg(\textstyle\frac{2}{i}\sum_{n=0}^{i-1}\left\|\textstyle\sum_{j=1}^{m_n}C_n\AK_{(n-1,n-j+1)}\omega_{n-j}\right\|\times\nonumber\\
    &~\times\left\|\textstyle\sum_{j=m_n+1}^{n}C_n\AK_{(n-1,n-j+1)}\omega_{n-j}\right\|>\textstyle\frac{\epsilon}{3}\Bigg)\leq\nonumber\\
    &\leq \mathds{P}\left(\sqrt{\textstyle\frac{2}{i}\sum_{n=0}^{i-1}\left\|\textstyle\sum_{j=1}^{m_n}C_n\AK_{(n-1,n-j+1)}\omega_{n-j}\right\|^2}\times\right.\nonumber\\
    &~\times\left.\sqrt{\textstyle\frac{2}{i}\sum_{n=0}^{i-1}\left\|\textstyle\sum_{j=m_n+1}^{n}C_n\AK_{(n-1,n-j+1)}\omega_{n-j}\right\|^2}>\textstyle\frac{\epsilon}{3}\right)\leq\nonumber\\
    &\leq \mathds{P}\left(\textstyle\frac{2}{i}\sum_{n=0}^{i-1}\left\|\textstyle\sum_{j=1}^{m_n}C_n\AK_{(n-1,n-j+1)}\omega_{n-j}\right\|^2>\frac{\epsilon}{3}\right)+\nonumber\\
    &~+\mathds{P}\left(\textstyle\frac{2}{i}\sum_{n=0}^{i-1}\left\|\textstyle\sum_{j=m_n+1}^{n}C_n\AK_{(n-1,n-j+1)}\omega_{n-j}\right\|^2>\frac{\epsilon}{3}\right),\label{eq:min_m_center_bound}
\end{align}
where the first inequality comes from applying the Cauchy Schwarz Inequality and the second inequality comes from applying Theorem \ref{thm:split_prob_mult}.
Then since 
\begin{align}
    \mathds{P}\left(\textstyle\frac{1}{i}\sum_{n=0}^{i-1}\left\|\textstyle\sum_{j=1}^{m_n}C_n\AK_{(n-1,n-j+1)}\omega_{n-j}\right\|^2>\frac{\epsilon}{3}\right)\leq\nonumber\\
    \leq\mathds{P}\left(\textstyle\frac{2}{i}\sum_{n=0}^{i-1}\left\|\textstyle\sum_{j=1}^{m_n}C_n\AK_{(n-1,n-j+1)}\omega_{n-j}\right\|^2>\frac{\epsilon}{3}\right)
\end{align}
and
\begin{align}
    \mathds{P}\left(\textstyle\frac{1}{i}\sum_{n=0}^{i-1}\left\|\textstyle\sum_{j=m_n+1}^{n}C_n\AK_{(n-1,n-j+1)}\omega_{n-j}\right\|^2>\frac{\epsilon}{3}\right)\leq\nonumber\\
    \leq\mathds{P}\left(\textstyle\frac{2}{i}\sum_{n=0}^{i-1}\left\|\textstyle\sum_{j=m_n+1}^{n}C_n\AK_{(n-1,n-j+1)}\omega_{n-j}\right\|^2>\frac{\epsilon}{3}\right)\label{eq:min_m_remain_prob_bound},
\end{align}
we can combine \eqref{eq:min_m_full_expand} with \eqref{eq:min_m_center_bound}-\eqref{eq:min_m_remain_prob_bound} to obtain
\begin{align}
    &\tau<2\mathds{P}\left(\textstyle\frac{1}{i}\sum_{n=0}^{i-1}\left\|\textstyle\sum_{j=1}^{m_n}C_n\AK_{(n-1,n-j+1)}\omega_{n-j}\right\|^2>\frac{\epsilon}{6}\right)+\nonumber\\
    &+2\mathds{P}\left(\textstyle\frac{1}{i}\sum_{n=0}^{i-1}\left\|\textstyle\sum_{j=m_n+1}^{n}C_n\AK_{(n-1,n-j+1)}\omega_{n-j}\right\|^2>\frac{\epsilon}{6}\right).\label{eq:min_m_prob_split}
\end{align}
If we can upper bound the second term in the RHS or \eqref{eq:min_m_prob_split} by $\frac{\tau}{2}$ the first term must be lower bounded by $\frac{\tau}{2}$ completing the proof. 
To provide this bound we make use of Markov's Inequality.
To this end, we first bound the expectation

\begin{align}
    &\mathds{E}\left(\textstyle\frac{1}{i}\sum_{n=0}^{i-1}\left\|\textstyle\sum_{j=m_n+1}^n C_n\AK_{(n-1,n-j+1)}\omega_{n-j}\right\|^2\right)=\nonumber\\
    &=\mathds{E}\Bigg(\textstyle\frac{1}{i}\sum_{n=0}^{i-1}\textstyle\sum_{j=m_n+1}^n\left(C_n\AK_{(n-1,n-j+1)}\omega_{n-j}\right)^\T\times\nonumber\\
    &\qquad\times\left(C_n\AK_{(n-1,n-j+1)}\omega_{n-j}\right)\Bigg)\leq\nonumber\\
    &\leq\mathds{E}\Big(\textstyle\frac{1}{i}\sum_{n=0}^{i-1}\textstyle\sum_{j=m+1}^\infty\left(C_{n+j}\AK_{(n+j-1,n+1)}\omega_{n}\right)^\T\times\nonumber\\
    &\qquad\times\left(C_{n+j}\AK_{(n+j-1,n+1)}\omega_{n}\right)\Big)\nonumber\\
    &\leq\textstyle\frac{1}{i}\sum_{n=0}^{i-1}\textstyle\sum_{j=m+1}^\infty p\eta_C^2\eta_{A1}^{2(j-1)}\eta_\omega^2=\frac{p\eta_C^2\eta_{A1}^{2m}\eta_\omega^2}{1-\eta_{A1}^2},
\end{align}
where the the first equality comes from expanding the norm and ignoring uncorrelated terms, the first inequality comes from rearranging the summation and allowing the second summation to go to infinity, the second inequality comes from distributing the expectation and upper bounding each element, and the final equality comes from evaluating the summations.
Since $\eta_{A1}<1$, we can choose $m$ sufficiently large such that 
\begin{align}
    \textstyle\frac{p\eta_C^2\eta_{A1}^{2m}\eta_\omega^2}{1-\eta_{A1}^2}<\frac{\tau\epsilon}{24}.
\end{align}
Using Markov's inequality \cite[Equation 5.31]{Billingsley1995ProbabilityMeasure} we have that 
\begin{align}
    \mathds{P}&\left(\textstyle\frac{1}{i}\sum_{n=0}^{i-1}\left\|\textstyle\sum_{j=m_n+1}^{n}C_n\AK_{(n-1,n-j+1)}\omega_{n-j}\right\|^2>\frac{\epsilon}{6}\right)\leq\nonumber\\
    &\textstyle\qquad\leq\frac{6p\eta_C^2\eta_{A1}^{2m}\eta_\omega^2}{(1-\eta_{A1}^2)\epsilon}<\frac{\tau}{4}
\end{align}
which completes the proof for the existence of $m$. 

Next, we prove the existence of $m^\prime$. 
Consider the expansion of \eqref{eq:min_m_claim}
\begin{align}
    &\plim{i}\textstyle\frac{1}{i}\sum_{n=0}^{i-1}C_n\sum_{j=1}^{m_n}\sum_{k=1}^{m_n}\AK_{(n-1,n-j+1)}\omega_{n-j}\nonumber\times\\
    &\hspace{1in}\times\omega_{n-k}^\T\AK_{(n-1,n-k+1)}^\T C_n^\T\neq0.
\end{align}
Considering the summands where $j\neq k$ we have that
\begin{align}
    &\plim{i}\textstyle\frac{1}{i}\sum_{n=0}^{i-1}C_n\AK_{(n-1,n-j+1)}\omega_{n-j}\nonumber\times\\
    &\hspace{1in}\times\omega_{n-k}^\T\AK_{(n-1,n-k+1)}^\T C_n^\T=0
\end{align}
by Theorem \ref{thm:converge_expected} since $\omega_n$ is independent and the dynamics are bounded and stable. 
If we further assume that there does not exist an $m^\prime$ for which \eqref{eq:min_m_claim_single} holds then by Theorem \ref{thm:split_prob_addition} we have that \eqref{eq:min_m_claim} does not hold which is a contradiction. 
Therefore, the set of integers less that or equal to $m$ for which \eqref{eq:min_m_claim_single} holds, is a non-empty finite set. 
The smallest element of this set then satisfies the conditions for $m^\prime$.
\end{proof}

Now returning to prove the Theorem.
\begin{proof}(Theorem \ref{thm:st_noise}) 
WLOG, in this proof, we allow summations to reference variables with negative index by assuming these values to be $0_r$ to ease notation. 
Assume that \eqref{eq:attack_additive_covariance} holds but \eqref{eq:att_equiv_attack_state_claim} does not.
Since \eqref{eq:att_equiv_attack_state_claim} does not hold, $m^\prime$ be chosen such that it satisfies the description in Lemma \ref{thm:min_m}. 
From Theorem \ref{thm:me_noise} we have that \eqref{eq:attack_additive_covariance} implies \eqref{eq:me_noise_state_noise_claim}.
Expanding the LHS of \eqref{eq:me_noise_state_noise_claim} using \eqref{eq:false_state_update} gives us
\begin{align}
    &\plim{i}\textstyle\frac{1}{i}\sum_{n=0}^{i-1}(C_n\delAttack_n-C_n\xi_n)(C_n\delAttack_n-C_n\xi_n)^\T=\nonumber\\
    &=\plim{i}\textstyle\frac{1}{i}\sum_{n=0}^{i-1}\Big(C_n\left(\delAttack_n-\sum_{j=1}^n\AK_{(n-1,n-j+1)}\omega_{n-j}\right)\times\nonumber\\
    &\quad\times\left(\delAttack_n-\textstyle\sum_{k=1}^n\AK_{(n-1,n-k+1)}\omega_{n-k}\right)^\T C_n^\T\Big)=0_r. \label{eq:st_noise_expand_attack_state}
\end{align}
By separating the index $m^\prime$ we can write
\begin{align}
    &\plim{i}\textstyle\frac{1}{i}\sum_{n=0}^{i-1}(C_n\delAttack_n-C_n\xi_n)(C_n\delAttack_n-C_n\xi_n)^\T=\nonumber\\
    &=\plim{i}\textstyle\frac{1}{i}\sum_{n=0}^{i-1}C_n\left(\delAttack_n-\sum_{\substack{1\leq j\leq n\\j\neq m^\prime}}\AK_{(n-1,n-j+1)}\omega_{n-j}\right)\times\nonumber\\
    &\quad\times\left(\delAttack_n-\textstyle\sum_{\substack{0\leq k\leq n\\k\neq m^\prime}}\AK_{(n-1,n-k+1)}\omega_{n-k}\right)^\T C_n^\T+\nonumber\\
    &\quad-C_n\left(\delAttack_n-\textstyle\sum_{\substack{1\leq j\leq n\\j\neq m^\prime}}\AK_{(n-1,n-j+1)}\omega_{n-j}\right)\times\nonumber\\
    &\quad\times\omega_{n-m^\prime}^\T\AK_{(n-1,n-m^\prime+1)}^\T C_n^\T-C_n\AK_{(n-1,n-m^\prime+1)}\omega_{n-m^\prime}\times\nonumber\\
    &\quad\times\left(\delAttack_n-\textstyle\sum_{\substack{0\leq k\leq n\\k\neq m^\prime}}\AK_{(n-1,n-k+1)}\omega_{n-k}\right)^\T C_n^\T+\nonumber\\
    &\quad+C_n\AK_{(n-1,n-m^\prime+1)}\omega_{n-m^\prime}\times\nonumber\\
    &\quad\times\omega_{n-m^\prime}^\T\AK_{(n-1,n-m^\prime+1)}^\T C_n^\T=0_r.\label{eq:st_noise_split_sum}
\end{align}
\begin{figure*}
    \centering
    \includegraphics[trim={1.6in 0.27in 1.6in 0.72in},clip,width=\textwidth]{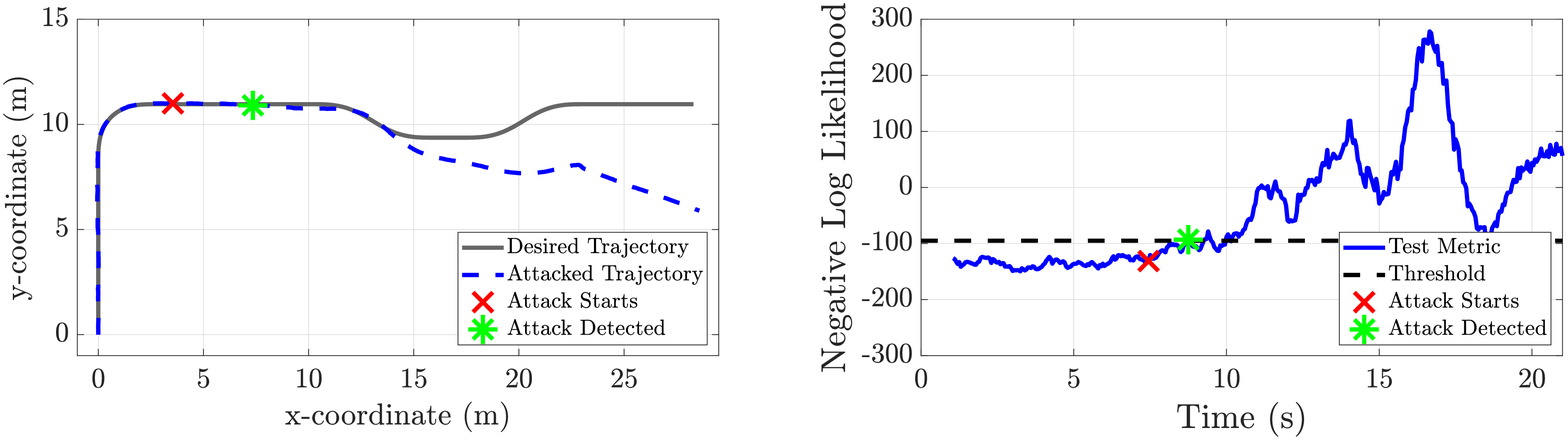}
    \caption{Desired and attacked trajectory of an LTV car model showing attack start and detection (Left); Corresponding LTV Dynamic Watermarking test metric showing attack start and detection (right)}
    \label{fig:LTV_attacked}
\end{figure*}
For now suppose that 
\begin{align}
    &\plim{i}\textstyle\frac{1}{i}\sum_{n=0}^{i-1}-C_n\left(\delAttack_n-\sum_{\substack{1\leq j\leq n\\j\neq m^\prime}}\AK_{(n-1,n-j+1)}\omega_{n-j}\right)\times\nonumber\\
    &\qquad\times\omega_{n-m^\prime}^\T\AK_{(n-1,n-m^\prime+1)}^\T C_n^\T=0_r.\label{eq:st_noise_cross_cancel}
\end{align}
Then by Theorem \ref{thm:split_prob_limit} we have that 
\begin{align}
    &\plim{i}\textstyle\frac{1}{i}\sum_{n=0}^{i-1}(C_n\delAttack_n-C_n\xi_n)(C_n\delAttack_n-C_n\xi_n)^\T=\nonumber\\
    &=\plim{i}\textstyle\frac{1}{i}\sum_{n=0}^{i-1}C_n\left(\delAttack_n-\sum_{\substack{1\leq j\leq n\\j\neq m^\prime}}\AK_{(n-1,n-j+1)}\omega_{n-j}\right)\times\nonumber\\
    &\qquad\times\left(\delAttack_n-\textstyle\sum_{\substack{0\leq k\leq n\\k\neq m^\prime}}\AK_{(n-1,n-k+1)}\omega_{n-k}\right)^\T C_n^\T+\nonumber\\
    &\qquad+C_n\AK_{(n-1,n-m^\prime+1)}\omega_{n-m^\prime}\times\nonumber\\
    &\qquad\times\omega_{n-m^\prime}^\T\AK_{(n-1,n-m^\prime+1)}^\T C_n^\T=0_r.\label{eq:st_noise_split_sum_cancel}
\end{align}
Furthermore, by our choice of $m^\prime$ we have that 
\begin{align}
    &\plim{i}\textstyle\frac{1}{i}\sum_{n=0}^{i-1}C_n\AK_{(n-1,n-m^\prime+1)}\omega_{n-m^\prime}\times\nonumber\\
    &\qquad\times\omega_{n-m^\prime}^\T\AK_{(n-1,n-m^\prime+1)}^\T C_n^\T\neq0_r,\label{eq:st_noise_finite_sum_nonconv}
\end{align} 
and since the terms are all positive-semidefinite matrices
\begin{align}
    &\mathds{P}\big(\big\|\textstyle\frac{1}{i}\sum_{n=0}^{i-1}C_n\AK_{(n-1,n-m^\prime+1)}\omega_{n-m^\prime}\times\nonumber\\
    &\qquad\times\omega_{n-k}^\T\AK_{(n-1,n-k+1)}^\T C_n^\T\big\|>\epsilon\big)\leq\nonumber\\
    &\leq\mathds{P}\Bigg(\Bigg\|\textstyle\frac{1}{i}\sum_{n=0}^{i-1}C_n\left(\delAttack_n-\sum_{\substack{1\leq j\leq n\\j\neq m^\prime}}\AK_{(n-1,n-j+1)}\omega_{n-j}\right)\times\nonumber\\
    &\qquad\times\left(\delAttack_n-\textstyle\sum_{\substack{0\leq k\leq n\\k\neq m^\prime}}\AK_{(n-1,n-k+1)}\omega_{n-k}\right)^\T C_n^\T+\nonumber\\
    &\qquad+C_n\AK_{(n-1,n-m^\prime+1)}\omega_{n-m^\prime}\times\nonumber\\
    &\qquad\times\omega_{n-m^\prime}^\T\AK_{(n-1,n-m^\prime+1)}^\T C_n^\T\Bigg\|>\epsilon\Bigg).\label{eq:st_noise_prob_ineq}
\end{align}
This implies that \eqref{eq:st_noise_split_sum} cannot hold which contradicts \eqref{eq:attack_additive_covariance}. 
Therefore \eqref{eq:att_equiv_attack_state_claim} must hold since otherwise there exists an $m^\prime$ satisfying Lemma \ref{lm:nonconv}. 

To complete the proof, we now show that \eqref{eq:st_noise_cross_cancel} indeed holds.
by Corollary \ref{cor:combine_prob_limit} this is equivalent to proving
\begin{align}
    &\plim{i}\textstyle\frac{1}{i}\sum_{n=0}^{i-1}C_n\sum_{\substack{1\leq j\leq n\\j\neq m^\prime}}\AK_{(n-1,n-j+1)}\omega_{n-j}\times\nonumber\\
    &\qquad\times\omega_{n-m^\prime}^\T\AK_{(n-1,n-m^\prime+1)}^\T C_n^\T=0_r\label{eq:st_noise_delta_cross_converge1}\\
    \intertext{and}
    &\plim{i}\textstyle\frac{1}{i}\sum_{n=0}^{i-1}-C_n\delAttack_n\omega_{n-m^\prime}^\T\AK_{(n-1,n-m^\prime+1)}^\T C_n^\T=0_r.\label{eq:st_noise_delta_cross_converge2}
\end{align}
Note, \eqref{eq:st_noise_delta_cross_converge1} holds by Corollary \ref{cor:converge_expected_cor} since all $\omega_n$ are mutually independent, $\|C_n\AK_{(n-1,n-m^\prime+1)}\|\leq\|C_n\|<\eta_C$, and \add{the auto-correlation is bounded by \eqref{eq:autocor_bound_1} in the appendix.}
\remove{
\begin{align}
    &\Bigg\|\mathds{E}\Bigg[\Bigg(\textstyle\sum_{\substack{1\leq j\leq n\\j\neq m^\prime}}\AK_{(n-1,n-j+1)}\omega_{n-j}\Bigg)\times\nonumber\\
    &\qquad\times\Bigg(\textstyle\sum_{\substack{1\leq k\leq n+i\\k\neq m^\prime}}\AK_{(n+i-1,n+i-k+1)}\omega_{n+i-k}\Bigg)^\T\Bigg]\Bigg\|=\nonumber\\
    &=\Bigg\|\textstyle\sum_{\substack{1\leq j\leq n\\j\neq m^\prime}}\AK_{(n-1,n-j+1)}\Sigma_{\omega,n-j}\AK_{(n+i-1,n-j+1)}^\T \Bigg\|\leq\nonumber\\
    &\leq\textstyle\sum_{j=1}^\infty \eta_{A1}^{2j-4+i}\eta_\omega=\textstyle\frac{\eta_{A2}^{i-2}\eta_\omega}{1-\eta_{A1}^2}.\label{eq:autocor_bound_1}
\end{align}
Here, the first equality comes from evaluating the expectation, the inequality comes from distributing the norm using triangle inequality and the subadditivity of the spectral norm, bounding the individual terms, and allowing the summation to include $j=m^\prime$ and go to infinity, and the final inequality comes from evaluating the summation.
}
Furthermore, expanding the LHS of \eqref{eq:st_noise_delta_cross_converge2} using \eqref{eq:observer_error_update_attack} gives us
\begin{align}
    &\plim{i}\textstyle\frac{1}{i}\sum_{n=0}^{i-1}-C_n\delAttack_n\omega_{n-m^\prime}^\T \AK_{(n-1,n-m^\prime+1)}^\T C_n^\T=\nonumber\\
    &=\plim{i}\textstyle\frac{1}{i}\sum_{n=0}^{i-1}C_n\big(\sum_{j=0}^{n-1}\AL_{(n-1,j+1)} L_j\zeta_j\nonumber\\
    &\quad+\textstyle\sum_{k=1}^{n}\AL_{(n-1,n-k+1)} L_{n-k}C_{n-k}\times\nonumber\\
    &\quad\times\textstyle\sum_{\ell=k+1}^n\AK_{(n-k-1,n-\ell+1)}\omega_{n-\ell}\big)\times\nonumber\\
    &\quad\times\omega_{n-m^\prime}^\T \AK_{(n-1,n-m^\prime+1)}^\T C_n^\T=0_r.\label{eq:st_noise_delta_cross_converge2_expand}
\end{align}
To prove that \eqref{eq:st_noise_delta_cross_converge2} holds, we use Corollary \ref{cor:combine_prob_limit} on \eqref{eq:st_noise_delta_cross_converge2_expand} and show that each term converges to $0_r$.
Note, by Theorem \ref{thm:converge_expected},
\begin{align}
    &\plim{i}\textstyle\frac{1}{i}\sum_{n=0}^{i-1}C_n\big(\sum_{j=0}^{n-1}\AL_{(n-1,j+1)} L_j\zeta_j\big)\times\nonumber\\
    &\qquad\times\omega_{n-m^\prime}^\T \AK_{(n-1,n-m^\prime+1)}^\T C_n^\T=0_r,
\end{align}
since $\|C_n\AK_{(n-1,n-m^\prime+1)}\|\leq\eta_C$, $\zeta_n$ and $\omega_n$ are mutually independent, and \add{the auto-correlation is bounded by \eqref{eq:autocor_bound_2} in the appendix.}
\remove{
\begin{align}
    &\left\|\mathds{E}\left[\textstyle\sum_{j=0}^{n-1}\sum_{k=0}^{n+i-1}\AL_{(n-1,j+1)} L_j\zeta_j\zeta_k^\T L_k^\T\AL_{(n+i-1,k+1)}^\T\right]\right\|=\nonumber\\
    &\quad=\left\|\textstyle\sum_{j=0}^{n-1}\AL_{(n-1,j+1)} L_j\Sigma_{\zeta_j} L_j^\T\AL_{(n+i-1,j+1)}^\T\right\|\leq\nonumber\\
    &\quad\leq\textstyle\sum_{j=0}^{n-1}\eta_{A2}^{2(n-1-j)}\eta_{A2}^i\eta_L^2\eta_{\zeta}\leq\frac{\eta_{A2}^i\eta_L^2\eta_\zeta}{1-\eta_{A2}^2}.\label{eq:autocor_bound_2}
\end{align}
}
Furthermore, considering the portion of $\delAttack_n$ not dependent on $\omega_{n-m^\prime}$, by Theorem \ref{thm:converge_expected},
\begin{align}
    &\plim{i}\textstyle\frac{1}{i}\sum_{n=0}^{i-1}C_n\sum_{j=1}^{n}\AL_{(n-1,n-j+1)} L_{n-j}C_{n-j}\times\nonumber\\
    &\quad\times\textstyle\sum_{\substack{k=j+1\\k\neq m^\prime}}^n\AK_{(n-j-1,n-k+1)}\omega_{n-k}\times\nonumber\\
    &\quad\times\omega_{n-m^\prime}^\T \AK_{(n-1,n-m^\prime+1)}^\T C_n^\T=0_r,
\end{align}
since $\omega_n$ are independent, $\|C_n\AK_{(n-1,n-m^\prime+1)}\|\leq\eta_C$, and \add{the auto-correlation is bounded by \eqref{eq:autocor_bound_3} in the appendix.}
Now if 
\begin{align}
    &\plim{i}\textstyle\frac{1}{i}\sum_{n=0}^{i-1}C_n\sum_{j=1}^{m^\prime-1}\AL_{(n-1,n-j+1)} L_{n-j}C_{n-j}\times\nonumber\\
    &\quad\times\AK_{(n-j-1,n-m^\prime+1)}\omega_{n-m^\prime}\times\nonumber\\
    &\quad\times\omega_{n-m^\prime}^\T \AK_{n-1,n-k+1}^\T C_n^\T=0_r,\label{eq:st_noise_delta_correlated_portion}
\end{align}
we have completed the proof.
To show this, we show that the trace of the matrix converges to 0 for each value of $j$.
\begin{align}
    &\plim{i}\textstyle\frac{1}{i}\sum_{n=0}^{i-1}\big(\omega_{n-m^\prime}^\T\AK_{(n-1,n-m^\prime+1)}^\T C_n^\T\times\nonumber\\
    &\hspace{0.5in}\times C_n\AL_{(n-1,n-j+1)}L_{n-j}C_{n-j}\times\nonumber\\
    &\hspace{0.5in}\times\AK_{(n-j-1,n-m^\prime+1)}\omega_{n-m^\prime}\big)\leq\nonumber\\
    &\leq\plim{i}\big(\textstyle\frac{1}{i}\sum_{n=0}^{i-1}\left\|C_n\AK_{(n-1,n-m^\prime+1)}\omega_{n-m^\prime}\right\|^2\big)^{1/2}\times\nonumber\\
    &\hspace{0.5in}\times\big(\textstyle\frac{1}{i}\sum_{n=0}^{i-1}\big\|C_n\AL_{(n-1,n-j+1)}L_{n-j}\times\nonumber\\
    &\hspace{0.5in}\times C_{n-j}\AK_{(n-j-1,n-m^\prime+1)}\omega_{n-m^\prime}\big\|^2\big)^{1/2}
\end{align}
where the inequality follow from the Cauchy Schwarz Inequality. 
Let $\epsilon,\tau>0$ be chosen arbitrarily.
Note that by Markov's Inequality
\begin{align}
    &\mathds{P}\Big(\textstyle\frac{1}{i}\sum_{n=0}^{i-1}\left\|C_n\AK_{(n-1,n-m^\prime+1)}\omega_{n-m^\prime}\right\|^2\geq\nonumber\\
    &\hspace{1in}\geq\textstyle\frac{2\eta_C^2\eta_{A1}^{2(m^\prime-1)}\eta_\omega}{1-\tau}\Big)\leq\nonumber\\
    &\hspace{0.05in}\leq\textstyle\frac{(1-\tau)\mathds{E}\left[\frac{1}{i}\sum_{n=0}^{i-1}\left\|C_n\AK_{(n-1,n-m^\prime+1)}\omega_{n-m^\prime}\right\|^2\right]}{2\eta_C^2\eta_{A1}^{2(m^\prime-1)}\eta_\omega}\leq\nonumber\\
    &\hspace{0.05in}\leq\textstyle\frac{(1-\tau)\eta_C^2\eta_{A1}^{2(m^\prime-1)}\eta_\omega}{2\eta_C^2\eta_{A1}^{2(m^\prime-1)}\eta_\omega}=\frac{1-\tau}{2}.
\end{align}
Furthermore by our choice of $m^\prime$, we have that there exists an $N$ such that $i>N$ implies
\begin{align}
    &\mathds{P}\Big(\textstyle\frac{1}{i}\sum_{n=0}^{i-1}\left\|C_{n-j}\AK_{(n-j-1,n-m^\prime+1)}\omega_{n-m^\prime}\right\|^2\leq\nonumber\\
    &\hspace{0.7in}\leq\textstyle\frac{\epsilon^2}{2\eta_C^4\eta_{A1}^{2(m^\prime-1)}\eta_{A2}^{2(j-1)}\eta_L^2\eta_\omega}\Big)\geq\frac{\tau+1}{2}.
\end{align}
Finally, applying Theorem \ref{thm:split_prob_mult_2}
\begin{align}
    &\mathds{P}\bigg(\bigg(\textstyle\frac{1}{i}\sum_{n=0}^{i-1}\left\|C_n\AK_{(n-1,n-m^\prime+1)}\omega_{n-m^\prime}\right\|^2\bigg)^{1/2}\times\nonumber\\
    &\hspace{0.4in}\times\bigg(\textstyle\frac{1}{i}\sum_{n=0}^{i-1}\big\|C_n\AL_{(n-1,n-j+1)}L_{n-j}\times\nonumber\\
    &\hspace{0.4in}\times C_{n-j}\AK_{(n-j-1,n-m^\prime+1)}\omega_{n-m^\prime}\big\|^2\bigg)^{1/2}\leq\epsilon\bigg)\geq\nonumber\\
    &\geq\mathds{P}\bigg(\textstyle\frac{1}{i}\sum_{n=0}^{i-1}\left\|C_{n-j}\AK_{(n-j-1,n-m^\prime+1)}\omega_{n-m^\prime}\right\|^2\nonumber\\
    &\hspace{0.4in}\leq\textstyle\frac{\epsilon^2}{2\eta_C^4\eta_{A1}^{2(m^\prime-1)}\eta_{A2}^{2(j-1)}\eta_L^2\eta_\omega}\bigg)+\nonumber\\
    &\hspace{0.4in}+\bigg(1-\mathds{P}\bigg(\textstyle\frac{1}{i}\sum_{n=0}^{i-1}\big\|C_n\AK_{(n-1,n-m^\prime+1)}\times\nonumber\\
    &\hspace{0.4in}\times\omega_{n-m^\prime}\big\|^2\geq 2\eta_C^2\eta_{A1}^2\eta_\omega\bigg)\bigg)-1\geq\\
    &\geq \frac{\tau+1}{2}+1-\frac{1-\tau}{2}-1=\tau.
\end{align}
Therefore \eqref{eq:st_noise_delta_correlated_portion} must hold.
\end{proof}

Having proven several intermediate results, we are now able to formally prove Theorem \ref{thm:LTV_Asymptotic_Main_Result}.
\begin{proof}(Theorem \ref{thm:LTV_Asymptotic_Main_Result})
When no attack is present, \eqref{eq:ltv_watermark_correlation_test} holds using Theorem \ref{thm:alpha} since $\alpha$ is equal to 0. 
Furthermore, \eqref{eq:ltv_covariance_test} holds since $\delta=\overline{\delta}$.

Now assume that an attack of non-zero asymptotic power is present and consider the following cases.\\
Case 1 ($\alpha\neq0$): Using Theorem \ref{thm:alpha}, \eqref{eq:ltv_watermark_correlation_test} does not hold.\\
Case 2 ($\alpha=0$): Note, \eqref{eq:ltv_covariance_test} implies zero asymptotic attack power as follows.
\begin{align}
    \eqref{eq:ltv_covariance_test}\underset{\text{Thm. \ref{thm:equiv_attack_covariance}}}{\Longleftrightarrow}\eqref{eq:attack_additive_covariance}~
    \begin{split}
    \underset{\text{Thm. \ref{thm:me_noise}}}{\Longrightarrow}\eqref{eq:att_equiv_attack_noise_claim}\\
    \underset{\text{Thm. \ref{thm:st_noise}}}{\Longrightarrow}\eqref{eq:att_equiv_attack_state_claim}
    \end{split}
~\underset{\text{Thm. \ref{thm:att_equiv}}}{\Longleftrightarrow}
    \left(\begin{matrix}
    \text{zero asymptotic}\\
    \text{attack power}
    \end{matrix}\right)\nonumber
\end{align}
Under our assumption of non-zero  asymptotic power, the contrapositive implies that \eqref{eq:ltv_covariance_test} does not hold.
\end{proof}
\section{Implementable Statistical Tests}\label{sec:LTV_practical}
While Section \ref{sec:LTV_theory} provides a necessary background for LTV Dynamic Watermarking, infinite limits are not well suited for real time attack detection. 
This section derives a statistical test using a sliding window approach.
Let 
\begin{align}
    \psi_n&=\begin{bmatrix}V_n(C_n\hat{x}_n-y_n)\\ e_{n-1}\end{bmatrix}\\\intertext{and}
    Q_n&=[\psi_{n-\ell}~\hdots~\psi_{n}][\psi_{n-\ell}~\hdots~\psi_{n}]^\T.
\end{align}
where $\ell+1$ is the window size, $\ell\in\mathbb{N}$, and $\ell\geq q+r-1$.
Note, $\psi_n$ is asymptotically uncorrelated and identically distributed such that $\psi_n\sim\mathcal{N}(0_{q+r\times1},S)$, for $n=1,2,3,\cdots$ where
\begin{align}
    S=\begin{bmatrix}I_r & 0_{r\times q}\\0_{q\times r} & \Sigma_e\end{bmatrix}.
\end{align}
Therefore, under the assumption of no attack, the distribution of $Q_n$ approaches a Wishart distribution with $\ell+1$ degrees of freedom and scale matrix $S$ as $\ell$ goes to infinity.
Furthermore, for a generalized replay attack with non-zero asymptotic power, Theorem \ref{thm:LTV_Asymptotic_Main_Result} proves that the scale matrix for $Q_n$ is no longer $S$ since either \eqref{eq:ltv_watermark_correlation_test} or \eqref{eq:ltv_covariance_test} is not satisfied. 
The Wishart distribution can then be used to define a statistical test using the negative log likelihood of the scale matrix $S$ given the sampled matrix $Q_n$:
\begin{align}
\mathcal{L}(Q_n)=(q+r-\ell)\log(|Q_n|)+tr(S^{-1}Q_n).
\end{align}

In theory, if the process and measurement noise covariances $\Sigma_{w,n}$ and $\Sigma_{z,n}$ are known, $V_n$ can be calculated using \eqref{eq:observer_error_covariance_normal}-\eqref{eq:residual_normalizer}.
In practice, these covariances are difficult to estimate which can lead to error in the estimate of $V_n$. 
To reduce this error, $V_n$ can be directly estimated using an ensemble average of $i$ realizations such that 
\begin{align}
V_n\approx\left(\textstyle\frac{1}{i}\sum_{j=1}^{i}(C_n\hat{x}_n^{(j)}-y_n^{(j)})(C_n\hat{x}_n^{(j)}-y_n^{(j)})^\T\right)^{-1/2}
\end{align}
where the superscript $(j)$ is the index of the realization.
This approximation is appropriate since by the weak law of large numbers we have that when no attack is present
\begin{align}
    \plim{i}~\textstyle\frac{1}{i}\sum_{j=1}^{i}&(C_n\hat{x}_n^{(j)}-y_n^{(j)})(C_n\hat{x}_n^{(j)}-y_n^{(j)})^\T=\nonumber\\
    &=C_n\Sigma_{\delta,n}C_n^\T+\Sigma_{z,n}
\end{align}
and $V_n$ is defined as in \eqref{eq:residual_normalizer}.

\section{Simulated Results}\label{sec:sim}
To provide proof of concept, we use a simplified car model
\begin{align}
    \begin{bmatrix}x\\y\\\psi\\v\\\dot{\psi}\end{bmatrix}=
    \begin{bmatrix}v\cos(\psi)\\v\sin(\psi)\\\dot{\psi}\\a\\\ddot{\psi}\end{bmatrix},\label{eq:dubins}
\end{align}
where the car has ground plane coordinates $(x,y)$, heading $\psi$, forward velocity $v$, and angular velocity $\dot{\psi}$.
Using the desired trajectory shown in Figure \ref{fig:LTV_attacked}, \eqref{eq:dubins} is linearized and discretized using a step size of 0.05 and zero order hold on the current state and input.
The controller and observer for the resulting LTV system are found using a linear quadratic regulator (LQR) to stabilize the system.
Furthermore, the process and measurement noise covariances are chosen such that they scale linearly with the velocity.

To compare LTI and LTV Dynamic Watermarking, a time invariant matrix normalization factor is calculated using the average of the residual covariance, while the time-varying matrix normalization factor is calculated using \eqref{eq:observer_error_covariance_normal}-\eqref{eq:residual_normalizer}. 
For both cases, we run 100 simulations with a window size of 20 and calculate the test metric and the average test metric as shown in Figure \ref{fig:compare}.
Note, while the LTV Dynamic Watermarking metric remains consistent over the entire simulation, the LTI counterpart has a repeatable time-varying pattern.

\begin{figure}
    \centering
    \includegraphics[width=\columnwidth]{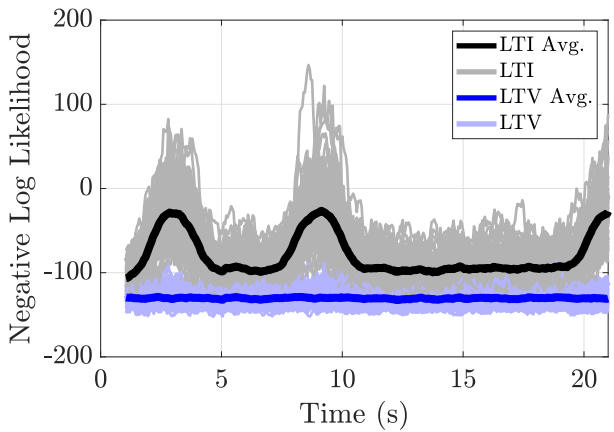}
    \caption{Simulated LTI and LTV Dynamic Watermarking test metrics for LTV car model under no attack}
    \label{fig:compare}
\end{figure}

Using the un-attacked data, a threshold for the LTV case is found such that the rate at which false alarms occur does not exceed once per every 50 seconds of run time.
Next consider an attack model satisfying \eqref{eq:state_update}-\eqref{eq:observer_error_update_attack}, with $\alpha$ equal to $-1$ and the measurement and process noise matching that of the true system.
The results of this attack on the system, and the ability of LTV Dynamic Watermarking to quickly detect it, are shown in Figure \ref{fig:LTV_attacked}.

\section{Conclusion} 
\label{sec:conclusion}
This paper derives Dynamic Watermarking for LTV systems, and provides asymptotic guarantees in addition to implementable tests.
A LTV generalized replay attack is defined and shown to be detectable by the Dynamic Watermarking method developed in this work.
Furthermore, a vehicle model with LTV Dynamic Watermarking is simulated to provide proof of concept of the implementable tests.
Using these simulations, the LTV Dynamic Watermarking is compared to its LTI counterpart and is shown to provide a more consistent test metric.

\bibliographystyle{IEEEtran}
\bibliography{collection}

\appendix
\section{Statistical Background}\label{sec:statistical background}
This section outlines the relevant background in statistics used in the paper\add{ and provides a few longer equations removed from proofs for readability}.
\subsection{Statistical Background}
First, we provide inequalities for functions of random variables using the following three theorems.
\begin{theorem}\label{thm:split_prob_addition}
Let $(a_i)_{i=1}^s$ be a finite set of random variables then 
\begin{align}
    \mathds{P}\left(\textstyle\sum_{i=1}^s a_i>\epsilon\right)\leq\textstyle\sum_{i=1}^s\mathds{P}\left(a_i>\frac{\epsilon}{s}\right).
\end{align}
\end{theorem}
\begin{proof}
Assume $a_i<\frac{\epsilon}{s}~\forall i$. 
This would imply that
\begin{align}
    \textstyle\sum_{i=1}^s a_i<\textstyle\sum_{i=1}^s\frac{\epsilon}{s}=\epsilon.
\end{align}
Therefore,
\begin{align}
    \left\{\textstyle\sum_{i=1}^s a_i>\epsilon\right\}\subseteq\textstyle\bigcup_{i=1}^s\left\{a_i>\frac{\epsilon}{s}\right\}
\end{align}
Furthermore,
\begin{align}
    \mathds{P}\left(\textstyle\sum_{i=1}^s a_i>\epsilon\right)&\leq\mathds{P}\left(\textstyle\bigcup_{i=1}^s\left\{a_i>\frac{\epsilon}{s}\right\}\right)\leq\nonumber\\
    &\leq\textstyle\sum_{i=1}^s\mathds{P}\left(a_i>\frac{\epsilon}{s}\right).
\end{align}
where the first inequality comes from the inclusion of the events and the final inequality comes from Boole's Inequality.
\end{proof}
\begin{theorem}\label{thm:split_prob_mult}
Let $(a_i)_{i=1}^s$ be a finite set of random variables then 
\begin{align}
    \mathds{P}\left(\textstyle\prod_{i=1}^s |a_i|>\epsilon\right)\leq\textstyle\sum_{i=1}^s\mathds{P}\left(|a_i|>\epsilon^{\frac{1}{s}}\right).
\end{align}
\end{theorem}
\begin{proof}
Assume $|a_i|<\epsilon^{\frac{1}{s}}~\forall i$. This would imply that
\begin{align}
    \textstyle\prod_{i=1}^s |a_i|<\textstyle\prod_{i=1}^s \epsilon^{\frac{1}{s}}=\epsilon.
\end{align}
The remainder of the proof follows closely to Theorem \ref{thm:split_prob_addition}.
\end{proof}

\begin{theorem}\label{thm:split_prob_mult_2}
Let $a$ and $b$ be random variables then for $\epsilon,\gamma>0$ we have
\begin{align}
    \mathds{P}\left(|ab|<\epsilon\right)\geq\mathds{P}\left(|a|<\gamma\right)+\mathds{P}\left(|b|<\textstyle\frac{\epsilon}{\gamma}\right)-1.
\end{align}
\end{theorem}

\begin{proof}
Note that 
\begin{align}\label{eq:split_prob_mult_bound1}
\mathds{P}(|ab|<\epsilon)\geq \mathds{P}\left(\{|a|<\gamma\}\cap\left\{|b|<\textstyle\frac{\epsilon}{\gamma}\right\}\right)
\end{align}
since $|a|<\gamma$ and $|b|<\epsilon/\gamma$ implies $|ab|<\epsilon$. 
By expanding the RHS of \eqref{eq:split_prob_mult_bound1} using inclusion exclusion and bounding the union term by 1, we get
\begin{align}
    &\mathds{P}(|ab|<\epsilon)
    \geq\mathds{P}(|a|<\gamma)+\mathds{P}\left(|b|<\textstyle\frac{\epsilon}{\gamma}\right)-1.
\end{align}
\end{proof}



It is often helpful to split a probabilistic limit into components of the underlying random variable. 
While this is not possible for all cases, we provide sufficient conditions here.
\begin{theorem}\label{thm:split_prob_limit}
Given sequences of random variables $a_i$ and $b_i$, and constants $a$ and $b$, suppose that 
\begin{align}
    \plim{i}~a_i+b_i=a+b ~\text{ and }~ \plim{i}~a_i=a\label{eq:split_prob_limit_suppose}
\end{align}
then
\begin{align}
    \plim{i}~b_i=b\label{eq:split_prob_limit_claim}.
\end{align}
\end{theorem}
\begin{proof}
Assume \eqref{eq:split_prob_limit_suppose} holds.
Given an $\epsilon>0$,
we have that
\begin{align}
    \mathds{P}&\left(\|b_i-b\|>\epsilon\right)\leq\nonumber
    \mathds{P}\left(\|a_i-a+b_i-b\|>\textstyle\frac{\epsilon}{2}\right)+\nonumber\\
    &\qquad+\mathds{P}\left(\|a_i-a\|>\textstyle\frac{\epsilon}{2}\right)\
\end{align}
where the inequality comes from triangle inequality and Theorem \ref{thm:split_prob_addition}.
Since both terms in this upper bound converge to zero, their sum, must as well.
Therefore, \eqref{eq:split_prob_limit_claim} must hold.
\end{proof}

Similarly we can combine probabilistic limits as follows.
\begin{cor}\label{cor:combine_prob_limit}
Consider sequences of random variables $a_i$ and $b_i$ and constants $a$ and $b$.
If
\begin{align}
    \plim{i}~b_i=b ~\text{ and }~ \plim{i}~a_i=a\label{eq:combine_prob_limit_suppose}
\end{align}
then 
\begin{align}
    \plim{i}~a_i+b_i=a+b.\label{eq:combine_prob_limit_claim}
\end{align}
\end{cor}
\begin{proof}Let $a_i^\prime=-a_i$, $a^\prime=-a$, $b_i^\prime=a_i+b_i$ and $b^\prime=a+b$.
Note, \eqref{eq:combine_prob_limit_suppose} implies \eqref{eq:split_prob_limit_suppose} is satisfied. 
Therefore, using Theorem \ref{thm:split_prob_limit}
\begin{align}
\plim{i}(a_i+b_i)=\plim{i} b_i^\prime=b^\prime=a+b.
\end{align}
\end{proof}

Since many of the limits in this paper deal with the average outer product of random vectors, it is important to know how and when these limits converge.
The following theorem provides sufficient conditions for convergence.
\begin{theorem}\label{thm:converge_expected}
Consider the sequences of vectors $(f_i)_{i=1}^\infty$ and $(g_i)_{i=1}^\infty$ where $f_i\sim\mathcal{N}(0_{s\times1},\Sigma_{f,i})$ and  $g_i\sim\mathcal{N}(0_{t\times1},\Sigma_{g,i})$.
Let $\eta$ and $\epsilon$ be scalar values such that $0<\eta<\infty$ and $\epsilon>1$.
If 
\begin{align}
    \|\mathds{E}[f_jf_{i}^\T]\|,~\|\mathds{E}[g_jg_{i}^\T]\|,~\|\mathds{E}[f_jg_{i}^\T]\|<\textstyle\frac{\eta}{\epsilon^{|i-j|}}\label{eq:converge_expected_autocorr_bound1},
\end{align}
$\forall~i,j\in\mathbb{N}$, then
\begin{align}
    \plim{i}\quad\textstyle\frac{1}{i}\sum_{j=1}^if_jg_j^\T-\mathds{E}[f_jg_j^\T]=0_{s\times t}.\label{eq:converge_expected_sum}
\end{align}
\end{theorem}
\begin{proof}
For \eqref{eq:converge_expected_sum} to hold, each of the element must also converge to 0 with probability 1. 
Therefore we will consider an arbitrary element and show it converges using an inequality derived from Chebyshev's inequality.
Selecting the element in an arbitrary row $m$ and column $n$ such that $0\leq m\leq s$ and $0\leq n\leq t$, let 
\begin{align}
    h_m^\T&=\begin{bmatrix}0_{1\times (m-1)}&1&0_{1\times(s-m)}\end{bmatrix}\\
    \intertext{and}
    h_n^\T&=\begin{bmatrix}0_{1\times (n-1)}&1&0_{1\times(t-n)}\end{bmatrix},
\end{align} 
then the sum for this single element can be written as
\begin{align}
    \rho_i=\textstyle\frac{1}{i}\sum_{j=1}^ih_m^\T f_ig_i^\T h_n-h_m^\T\mathds{E}[f_jg_j^\T]h_n.\label{eq:converge_expected_scalar_sum}
\end{align}
In order to use Chebyshev's inequality we must first bound the second moment of $\rho_i$. We start by expanding $\rho_i^2$ using  \eqref{eq:converge_expected_scalar_sum} and canceling like terms to get
\begin{align}
    \left|\mathds{E}[\rho_i^2]\right|
    &=\big|\textstyle\frac{1}{i^2}\sum_{j=1}^i\sum_{k=1}^i\mathds{E}[h_m^\T f_jg_j^\T h_nh_m^\T f_kg_k^\T h_n]+\nonumber\\
    &\qquad-h_m^\T\mathds{E}[f_jg_j^\T]h_nh_m^\T\mathds{E}[f_kg_k^\T]h_n\big|.
\end{align}
Expanding the expectation in the first term using \cite[Equation 2.3.8]{brillinger1981time} and once again canceling like terms results in 
\begin{align}
    \left|\mathds{E}[\rho_i^2]\right|
    &=\big|\textstyle\frac{1}{i^2}\sum_{j=1}^i\sum_{k=1}^i h_m^\T\mathds{E}[f_jg_k^\T]h_n h_m^\T\mathds{E}[f_kg_j^\T]h_n+\nonumber\\
    &\qquad+h_m^\T\mathds{E}[f_jf_k^\T]h_m h_n^\T\mathds{E}[g_jg_k^\T]h_n\big|.
\end{align}
Distributing the norm across the addition and multiplication using triangle inequality and the sub-multiplicative property of the 2 norm we then get the upper bound
\begin{align}
    &\left|\mathds{E}[\rho_i^2]\right|
    \leq \textstyle\frac{1}{i^2}\sum_{j=1}^i\sum_{k=1}^i \|h_m\|^2\|h_n\|^2\|\mathds{E}[f_jg_k^\T]\|\times\nonumber\\
    &\quad\times\|\mathds{E}[f_kg_j^\T]\|+\|h_m\|^2\|h_n\|^2\|\mathds{E}[f_jf_k^\T]\|~\|\mathds{E}[g_jg_k^\T]\|.
\end{align}
Applying the bounds in \eqref{eq:converge_expected_autocorr_bound1} and the fact that $\|h_m\|=\|h_n\|=1$ we can further upper bound resulting in
\begin{align}
    \left|\mathds{E}[\rho_i^2]\right|
    &\leq \textstyle\frac{1}{i^2}\sum_{j=1}^i\sum_{k=1}^i \frac{2\eta^2}{\epsilon^{2|j-k|}}\label{eq:converge_expected_finite_geometric}
\end{align}
Furthermore,
\begin{align}
    \left|\mathds{E}[\rho_i^2]\right|
    &\leq \textstyle\frac{1}{i^2}\sum_{j=1}^i\sum_{k=1}^\infty\frac{4\eta^2}{\epsilon^{2k}}=\frac{4\eta^2}{i(1-\frac{1}{\epsilon^2})}.\label{eq:converge_expected_inf_geometric}
\end{align}
where the inequality comes from the summation in \eqref{eq:converge_expected_inf_geometric} containing all of the summands in \eqref{eq:converge_expected_finite_geometric} and the fact that all summands are non-negative. 

Finally, using this bound and applying Chebyshev's Inequality \cite[Equation 5.32]{Billingsley1995ProbabilityMeasure} we have that, for an arbitrary choice of $\beta>0$,
\begin{align}
    P(|\rho_i|>\beta)\leq\textstyle\frac{\mathds{E}[\rho_i^2]}{\beta^2}=\frac{4\eta^2}{i\beta^2(1-\frac{1}{\epsilon^2})}.
\end{align}
Therefore, $\rho_i$ converges to 0 with probability 1.
Since the matrix element was chosen arbitrarily, \eqref{eq:converge_expected_sum} must hold.
\end{proof}

As a direct result of Theorem \ref{thm:converge_expected}, we can also make similar claims for Gaussian sequences that have been multiplied by bounded linear transformations.
\begin{cor}\label{cor:converge_expected_cor}
Consider a pair of sequences of vectors $(f_i)_{i=1}^\infty$ and $(g_i)_{i=1}^\infty$ where $f_i\sim\mathcal{N}(0_{s\times1},\Sigma_{f,i})$ and  $g_i\sim\mathcal{N}(0_{t\times1},\Sigma_{g,i})$.
Furthermore, consider the sequences of time varying matrices $(T_i)_{i=1}^\infty$ and $(U_i)_{i=1}^\infty$, where $T_i\in\mathbb{R}^{s^\prime\times s}$ and $U_i\in\mathbb{R}^{t^\prime\times t}$. 
Assume that 
\begin{align}
    \|T_i\|\leq \eta_T~\text{ and }~\|U_i\|\leq\eta_U.\label{eq:converge_expected_cor_mat_bound}
\end{align}
Let $\eta,\epsilon\in\mathbb{R}$ such that $0<\eta<\infty$ and $\epsilon>1$.
If 
\begin{align}
    \|\mathds{E}[f_jf_{i}^\T]\|,~
    \|\mathds{E}[g_jg_{i}^\T]\|,~
    \|\mathds{E}[f_jg_{i}^\T]\|&<\textstyle\frac{\eta}{\epsilon^{|i-j|}},\label{eq:converge_expected_autocorr_bound2}
\end{align}
$\forall~i,j\in\mathbb{N}$, then
\begin{align}
    \plim{i}\quad\textstyle\frac{1}{i}\sum_{j=1}^iT_jf_jg_j^\T U_j^\T-\mathds{E}T_j\left[f_jg_j^\T \right]U_j^\T=0_{s^\prime\times t^\prime}.\label{eq:converge_expected_cor_sum}
\end{align}
\end{cor}
\begin{proof} We prove this result by showing that the bounded linear transform generates new sequences that satisfy the conditions described in Theorem \ref{thm:converge_expected}. 
Let 
\begin{align}
    f_i^\prime=T_if_i~\forall i~\text{ and } g_i^\prime=U_ig_i~\forall i
\end{align}
then $f_i^\prime\sim\mathcal{N}(0_{s^\prime\times1},T_i\Sigma_{f,i}T_i^\T)$ and $g_i^\prime\sim\mathcal{N}(0_{t^\prime\times1},U_i\Sigma_{g,i}U_i^\T)$. 
Furthermore, we have that 
\begin{align}
    \|\mathds{E}[f_j^\prime f_{i}^{\prime\T}]\|
    \leq\|T_j\|\|T_{i}\|\|\mathds{E}[f_jf_{i}^\T]\|<\textstyle\frac{\eta_T^2\eta}{\epsilon^{|i-j|}}
\end{align}
where the first inequality comes from the submultiplicative property of the spectral norm and the second from applying \eqref{eq:converge_expected_autocorr_bound2} and 
\eqref{eq:converge_expected_cor_mat_bound}.
Similarly,
\begin{align}
    \textstyle\|\mathds{E}[g_j^\prime g_{i}^{\prime\T}]\|<\frac{\eta_U^2\eta}{\epsilon^{|i-j|}}
    \quad\text{and}\quad
    \|\mathds{E}[f_j^\prime g_{i}^{\prime\T}]\|<\frac{\eta_U\eta_T\eta}{\epsilon^{|i-j|}}.
\end{align}
Let $\eta^\prime=\max\{\eta_U^2\eta,\eta_T^2\eta,\eta_U\eta_T\eta\}$ and $\epsilon^\prime=\epsilon$ then
\begin{align}
    \textstyle\|\mathds{E}[f_j^\prime f_{i}^{\prime\T}]\|,~
    \|\mathds{E}[g_j^\prime g_{i}^{\prime\T}]\|,~
    \|\mathds{E}[f_j^\prime g_{i}^{\prime\T}]\|<\frac{\eta^\prime}{\epsilon^{\prime |i-j|}}
\end{align}
which satisfies the conditions for using Theorem \ref{thm:converge_expected} which implies that 
\begin{align}
    \plim{i}&\quad\textstyle\frac{1}{i}\sum_{j=1}^i f_j^\prime g_j^{\prime\T}-\mathds{E}[f_j^\prime g_j^{\prime\T}]=0_{s^\prime\times t^\prime},
\end{align}
which completes the proof since
\begin{align}
    f_j^\prime g_j^{\prime\T}-\mathds{E}[f_j^\prime g_j^{\prime\T}]=T_jf_jg_j^\T U_j^\T-T_j\mathds{E}[f_jg_j^\T] U_j^\T.
\end{align}
\end{proof}

To use Theorem \ref{thm:converge_expected} and Corollary \ref{cor:converge_expected_cor}, we provide sufficient conditions for a Gaussian sequence to satisfy conditions \eqref{eq:converge_expected_autocorr_bound1} and \eqref{eq:converge_expected_autocorr_bound2}.
\begin{theorem}\label{thm:bounded_expectations}
Consider the Gaussian process 
\begin{align}
    a_{i+1}=M_ia_i+b_i\label{eq:bounded_expectations_update}
\end{align}
where $a_0=0_{s\times1}$ and $b_i$ are independent gaussian distributed random vaiables such that $b_i\sim\mathcal{N}(0_{s\times1},\Sigma_{b,i})$. 
If $\exists \epsilon_1,\epsilon_2$ such that $\|M_i\|<\epsilon_1<1$ and $\|\Sigma_{b,i}\|<\epsilon_2<\infty~\forall~i$ then 
\begin{align}
    \left\|\mathds{E}\left[a_ja_{i}^\T\right]\right\|<\textstyle\frac{\eta}{\epsilon^{|i-j]}}\label{eq:bounded_expectations_auto_claim},
\end{align}
where $\eta=\frac{\epsilon_2}{1-\epsilon_1^2}$ and $\epsilon=\frac{1}{\epsilon_1}$.
\end{theorem}
\begin{proof}
Consider the  LHS of \eqref{eq:bounded_expectations_auto_claim} when $i=j$. 
We can expand $a_ja_j^T$ using \eqref{eq:bounded_expectations_update} iterativley to get 
\begin{align}
    &\|\mathds{E}[a_ja_j^\T]\|=\nonumber\\
    &~=\left\|\textstyle\sum_{i=1}^{j} M_{j-1}\hdots M_{j-i+1}\Sigma_{b,j-i} M_{j-i+1}^\T\hdots M_{j-1}^\T\right\|.
\end{align}
We upper bound this norm as follows
\begin{align}
    \|\mathds{E}[a_ja_j^\T]\|&\leq\textstyle\sum_{i=1}^{j}\|M_{j-1}\|\hdots\|M_{j-i+1}\|\times\nonumber\\
    &\qquad\times\|\Sigma_{b,j-i}\|~\|M_{j-i+1}^\T\|\hdots \|M_{j-1}^\T\|\nonumber\\
    &<\textstyle\sum_{i=1}^{i} \epsilon _2\epsilon_1^{2(j-1)}\leq\frac{\epsilon_2}{1-\epsilon_1^2}\label{eq:bounded_expectations_cov_proof},
\end{align}
where the first inequality comes from applying triangle inequality and the sub-multiplicative property of the spectral norm and the second inequality comes from applying the bounds on $\|M_i\|$ and $\|\Sigma_{b,i}\|$ and then bounding the resulting geometric series.

We now focus on \eqref{eq:bounded_expectations_auto_claim} for when $i\neq j$.
Consider the following which has been expanded using \eqref{eq:bounded_expectations_update}
\begin{align}
    &\|\mathds{E}[a_{j+i}a_j^\T]\|=\|\mathds{E}[a_ja_{j+i}^\T]\|=\|\mathds{E}[a_j(M_{j+i-1}\hdots M_{j}a_j+\nonumber\\
    &\qquad+\textstyle\sum_{k=1}^{i}M_{j+i-1}\hdots M_{j+i-k+1}b_{j+i-k})^\T]\|.
\end{align}
Since $\mathds{E}[a_jb_{j+i-k}]=0~\forall~k\leq i$, this simplifies to
\begin{align}
    \|\mathds{E}[a_{j+i}a_j^\T]\|&=\|\mathds{E}[a_ja_{j+i}^\T]\|\nonumber\\
    &=\|\mathds{E}[a_ja_j^\T]M_{j}^\T\hdots M_{j+i-1}^\T\|\textstyle<
    \frac{\eta}{\epsilon^i},
\end{align}
where the inequality comes from \eqref{eq:bounded_expectations_cov_proof} and $\|M_i\|<\epsilon_1$.
\end{proof}

Next, we show that, when $\alpha$ being equal to 0, the full system state satisfies the conditions of Theorem \ref{thm:bounded_expectations}.
\begin{theorem}\label{thm:alpha_zero_bounded_state}
Consider an attacked LTV system satisfying the dynamics in \eqref{eq:state_update}-\eqref{eq:observer_error_update_attack} and the attack model in \eqref{eq:attack_definition}-\eqref{eq:false_state_update}. 
Assume the attack scaling factor $\alpha$ is equal to 0. 
Then $\exists~\eta>0$ and $\epsilon>1$ such that  
\begin{align}
    \left\|\mathds{E}\left[\begin{bmatrix}x_{n}\\\bar{\delta}_{n}\\\hat{\delta}_{n}\\\xi_n\end{bmatrix}
    \begin{bmatrix}x_{n+i}\\\bar{\delta}_{n+i}\\\hat{\delta}_{n+i}\\\xi_{n+i}\end{bmatrix}^\T \right]\right\|<\textstyle\frac{\eta}{\epsilon^i}.\label{eq:alpha_zero_bounded_state_claim}
\end{align}
\end{theorem}

\begin{proof}
We prove this result using Theorem \ref{thm:bounded_expectations}. 
First note that using \eqref{eq:state_update}-\eqref{eq:observer_error_update_attack}, \eqref{eq:attack_definition}-\eqref{eq:false_state_update}, and assuming $\alpha=0$ we can write
\begin{align}
    \begin{bmatrix}x_{n+1}\\\bar{\delta}_{n+1}\\\hat{\delta}_{n+1}\\\xi_{n+1}\end{bmatrix}
    &=M_n\begin{bmatrix}x_{n}\\\bar{\delta}_{n}\\\hat{\delta}_{n}\\\xi_n\end{bmatrix}
    +T_n\begin{bmatrix}e_n\\w_n\\z_n\\\zeta_n\\\omega_n\end{bmatrix}
\end{align}
where
\begin{align}
    M_n=\begin{bmatrix}\AK_n&B_n K_n&B_n K_n&0_{p}\\0_{p}&\AL_n&0_{p}&0_{p}\\0_{p}&0_{p}&\AL_n&-L_n C_n\\0_{p}&0_{p}&0_{p}&\AK_n\end{bmatrix}\\
    \intertext{and}
    T_n=
    \begin{bmatrix}
    B_n & I_{p}& 0_{p\times r} & 0_{p\times r} & 0_p\\
    0_{p\times q} & -I_{p} & -L_n & 0_{p\times r} & 0_p\\
    0_{p\times q} & 0_p & 0_{p\times r} & -L_n & 0_p\\
    0_{p\times q} & 0_p & 0_{p\times r} &0_{p\times r} & I_{p}
    \end{bmatrix}.
\end{align}
Let $\epsilon_1=\max\{\eta_{A1},\eta_{A2}\}$ then 
\begin{align}
    \left\|M_n\right\|<\epsilon_1<1
\end{align}
since the eigenvalues of upper block diagonal matrices are the set of eigenvalues of the block elements on the diagonal and $\|\AK_n\|<\eta_{A1}<1$ and $\|\AL_n\|<\eta_{A2}<1$. 
Furthermore, denote
\begin{align}
    b_n=T_n
    \begin{bmatrix}e_n^\T&w_n^\T&z_n^\T&\zeta_n^\T&\omega_n^\T\end{bmatrix}^\T,
\end{align}
then $b_n\sim\mathcal{N}(0,\Sigma_{b,n})$ where 
\begin{align}
    \Sigma_{b,n}=T_n\begin{bmatrix}
    \Sigma_e & 0_{q\times p} & 0_{q\times r} & 0_{q\times r} & 0_{q\times p}\\
    0_{p\times q} & \Sigma_{w,n} & 0_{p\times r} & 0_{p\times r} & 0_p\\
    0_{r\times q} & 0_{r\times p} & \Sigma_{z,n} & 0_r & 0_{r\times p}\\
    0_{r\times q} & 0_{r\times p} & 0_r & \Sigma_{\zeta,n} & 0_{r\times p}\\
    0_{p\times q} & 0_p & 0_{p\times r} & 0_{p\times r} & \Sigma_{\omega,n}
    \end{bmatrix}T_n^\T.
\end{align}
Since $B_n,L_n,\Sigma_e,\Sigma_{w,n},\Sigma_{z,n},\Sigma_{\zeta,n},$ and $\Sigma_{\omega,n}$ are all bounded we have that $\|\Sigma_{b,n}\|<\epsilon_2$ for some $0\leq\epsilon_2<\infty$. 
Denoting 
\begin{align}
    a_n=\begin{bmatrix}x_{n}^\T&\bar{\delta}_{n}^\T&\hat{\delta}_{n}^\T&\xi_{n}\end{bmatrix}^\T,
\end{align}
we are able to complete the proof using Theorem \ref{thm:bounded_expectations}.
\end{proof}

Since the asymptotic attack power uses the inner product of $v_n$ while most other limits use outer products, we relate these limits in the following Lemma.
\begin{lm}\label{lm:mat_scal}
Consider a sequence of random vectors $(b_n)_{n=0}^\infty$ such that $b_n\in\mathbb{R}^s$.
\begin{align}
    \plim{i}~\textstyle\frac{1}{i}\sum_{n=0}^{i-1}b_nb_n^\T=0_s\\
    \intertext{if and only if}
    \plim{i}~\textstyle\frac{1}{i}\sum_{n=0}^{i-1}b_n^\T b_n=0.\label{eq:matrix_lim_eq_scalar_lim}
\end{align}
\end{lm}

\begin{proof}
Assume that the RHS of \eqref{eq:matrix_lim_eq_scalar_lim} holds. 
Note that 
\begin{align}
    \left\|\textstyle\frac{1}{i}\sum_{n=0}^{i-1}\displaystyle b_n b_n^\T\right\|\leq\textstyle\frac{1}{i}\sum_{n=0}^{i-1}\displaystyle \|b_n b_n^\T\|=\textstyle\frac{1}{i}\sum_{n=0}^{i-1}b_n^\T b_n.
\end{align}
where the inequality comes from triangle inequality and the equality comes from the matrix $b_nb_n^\T$ being singular.
This implies that  
\begin{align}
    \mathds{P}\left(\left|\textstyle\frac{1}{i}\sum_{n=0}^{i-1}\displaystyle b_n^\T b_n\right|>\epsilon\right)\geq\mathds{P}\left(\left\|\textstyle\frac{1}{i}\sum_{n=0}^{i-1}b_n b_n^\T\right\|>\epsilon\right).\label{eq:mat_scal_right_ineq}
\end{align} 
Since the LHS of \eqref{eq:mat_scal_right_ineq} converges to zero as $i\to \infty$ as a result of our assumption, the RHS must do so as well which directly implies the LHS of \eqref{eq:matrix_lim_eq_scalar_lim} holds.

Now assume that the LHS of \eqref{eq:matrix_lim_eq_scalar_lim} holds. 
Then since 
\begin{align}
    \textstyle\frac{1}{i}\sum_{n=0}^{i-1}b_n^\T b_n=\text{tr}\left(\frac{1}{i}\sum_{n=0}^{i-1}b_n b_n^\T\right),
\end{align}
and for the matrix to converge it must also converge element-wise, we have that the RHS of \eqref{eq:matrix_lim_eq_scalar_lim} also holds.
\end{proof}

Next, we show that if conditions such as \eqref{eq:ltv_covariance_test} do not hold, linear transforms of the limit also do not converge to zero given the conditions in the following lemma hold.
\begin{lm}\label{lm:nonconv}
Consider a family of matrices $R_n\in\mathbb{R}^{t\times s}$ with full column rank. 
Assume there exists $\eta\in\mathbb{R}$ such that $0<\eta\leq\lambda_n$, where $\lambda_n$ is the smallest eigenvalue of $R_n^TR_n$.
Furthermore, consider a sequence of random vectors $f_n\sim\mathcal{N}(0_{s\times 1},\Sigma_f)$ such that $\Sigma_{f,n}$ is positive semi-definite.
If  
\begin{align}
    \textstyle\sum_{i=1}^\infty\|\mathds{E}[f_nf_{n+i}^\T]\|<\infty~\forall n\label{eq:nonconv_original_claim1}
\end{align}
and
\begin{align}
    \plim{i}~\textstyle\frac{1}{i}\sum_{n=0}^{i-1} f_nf_n^\T \neq 0_s\label{eq:nonconv_original_claim2},
\end{align}
then
\begin{align}
    \plim{i}~\textstyle\frac{1}{i}\sum_{n=0}^{i-1} R_nf_nf_n^\T R_n^\T\neq 0_t\label{eq:nonconv_transform_claim}.
\end{align}
\end{lm}

\begin{proof}(Lemma \ref{lm:nonconv})
Assume that \eqref{eq:nonconv_original_claim1}-\eqref{eq:nonconv_original_claim2} holds, but 
\begin{align}
    \plim{i}~\textstyle\frac{1}{i}\sum_{n=0}^{i-1} R_nf_nf_n^\T R_n^\T= 0_t.
\end{align}
Applying Lemma \ref{lm:mat_scal} we have that 
\begin{align}
    \plim{i}~\textstyle\frac{1}{i}\sum_{n=0}^{i-1} f_n^\T R_n^\T R_nf_n= 0.
\end{align}
This implies that 
\begin{align}
    \plim{i}~\frac{\eta}{i}\textstyle\sum_{n=0}^{i-1} f_n^\T f_n=0
\end{align}
since $\eta f_n^\T f_n\leq\lambda_n f_n^\T f_n\leq f_n^\T R_n^\T R_n f_n$.
Since the limit is not affected by the constant $\eta$, and using Lemma \ref{lm:mat_scal}, this contradicts \eqref{eq:nonconv_original_claim2}.
Therefore, \eqref{eq:nonconv_transform_claim} must hold.
\end{proof}

\add{
\subsection{Ommited Equations}
The following equations were ommited from the proof of Theorem \ref{thm:st_noise} to improve readability.
First, note that
\begin{align}
    &\Bigg\|\mathds{E}\Bigg[\Bigg(\textstyle\sum_{\substack{1\leq j\leq n\\j\neq m^\prime}}\AK_{(n-1,n-j+1)}\omega_{n-j}\Bigg)\times\nonumber\\
    &\qquad\times\Bigg(\textstyle\sum_{\substack{1\leq k\leq n+i\\k\neq m^\prime}}\AK_{(n+i-1,n+i-k+1)}\omega_{n+i-k}\Bigg)^\T\Bigg]\Bigg\|=\nonumber\\
    &=\Bigg\|\textstyle\sum_{\substack{1\leq j\leq n\\j\neq m^\prime}}\AK_{(n-1,n-j+1)}\Sigma_{\omega,n-j}\AK_{(n+i-1,n-j+1)}^\T \Bigg\|\leq\nonumber\\
    &\leq\textstyle\sum_{j=1}^\infty \eta_{A1}^{2j-4+i}\eta_\omega=\textstyle\frac{\eta_{A2}^{i-2}\eta_\omega}{1-\eta_{A1}^2},\label{eq:autocor_bound_1}
\end{align}
where the first equality comes from evaluating the expectation, the inequality comes from distributing the norm using triangle inequality and the subadditivity of the spectral norm, bounding the individual terms, and allowing the summation to include $j=m^\prime$ and go to infinity, and the final inequality comes from evaluating the summation.
Using similar reasoning, we also have
\begin{align}
    &\left\|\mathds{E}\left[\textstyle\sum_{j=0}^{n-1}\sum_{k=0}^{n+i-1}\AL_{(n-1,j+1)} L_j\zeta_j\zeta_k^\T L_k^\T\AL_{(n+i-1,k+1)}^\T\right]\right\|=\nonumber\\
    &\quad=\left\|\textstyle\sum_{j=0}^{n-1}\AL_{(n-1,j+1)} L_j\Sigma_{\zeta_j} L_j^\T\AL_{(n+i-1,j+1)}^\T\right\|\leq\nonumber\\
    &\quad\leq\textstyle\sum_{j=0}^{n-1}\eta_{A2}^{2(n-1-j)}\eta_{A2}^i\eta_L^2\eta_{\zeta}\leq\frac{\eta_{A2}^i\eta_L^2\eta_\zeta}{1-\eta_{A2}^2}.\label{eq:autocor_bound_2}
\end{align}
and
\begin{align}
    &\Bigg\|\mathds{E}\Bigg[\Bigg(C_n\textstyle\sum_{j=1}^{n}\AL_{(n-1,n-j+1)} L_{n-j}C_{n-j}\times\nonumber\\
    &\qquad\times\textstyle\sum_{\substack{k=j+1\\k\neq m^\prime}}^n\AK_{(n-j-1,n-k+1)}\omega_{n-k}\Bigg)\times\nonumber\\
    &\qquad\times\Bigg(C_{n+i}\textstyle\sum_{j=1}^{n+i}\AL_{(n+i-1,n+i-j+1)} L_{n+i-j}C_{n+i-j}\times\nonumber\\
    &\qquad\times\textstyle\sum_{\substack{k=j+1\\k\neq m^\prime}}^{n+i}\AK_{(n+i-j-1,n+i-k+1)}\omega_{n+i-k}\Bigg)^\T\Bigg]\Bigg\|=\nonumber\\
    &=\Bigg\|\textstyle\sum_{j=1}^{n}\sum_{\ell=1}^{n+i}C_n\AL_{(n-1,n-j+1)} L_{n-j}C_{n-j}\times\nonumber\\
    &\qquad\times\textstyle\sum_{\substack{k=\max\{j+1,\ell+1\}\\k\neq m^\prime}}^n\AK_{(n-j-1,n-k+1)}\Sigma_{\omega,n-k}\times\nonumber\\
    &\qquad\times\AK_{(n+i-\ell-1,n-k+1)}^\T C_{n+i-\ell}^\T L_{n+i-\ell}^\T\times\nonumber\\
    &\qquad\times\AL_{(n+i-1,n+i-\ell+1)}^\T C_{n+i}^\T\Bigg\|\leq\nonumber\\
    &\leq\textstyle\sum_{j=1}^{n}\sum_{\ell=1}^{n}\eta_C^4\eta_L^2\eta_{A}^{\ell+j-2}\times\nonumber\\
    &\qquad\textstyle\times\sum_{\substack{k=\max\{j+1,\ell+1\}\\k\neq m^\prime}}^n\eta_{A}^{2k-j-\ell-2+i}\eta_\omega\leq\nonumber\\
    &\leq\eta_A^{i-4}\eta_C^4\eta_L^2\eta_\omega2\textstyle\sum_{j=1}^{\infty}\sum_{\ell=j}^\infty\sum_{k=\ell+1}^\infty \eta^{2k}=\textstyle\frac{2\eta_A^i\eta_C^4\eta_L^2\eta_\omega}{(1-\eta_A^2)^3}.\label{eq:autocor_bound_3}
\end{align}
}

\end{document}